\documentclass[a4paper,11pt]{amsart}

\usepackage{a4wide, amsmath, amsfonts, amssymb, mathrsfs, amsthm, bbm, color, stmaryrd, breqn}
\usepackage[hyperfootnotes=false]{hyperref} 

\allowdisplaybreaks[2]

\numberwithin{equation}{section}

\newtheorem{theorem}{Theorem}[section]
\newtheorem{lemma}[theorem]{Lemma}

\newtheorem{remark}[theorem]{Remark}
\newtheorem{proposition}[theorem]{Proposition}

\theoremstyle{definition}
\newtheorem{definition}[theorem]{Definition}

\newcommand{\dd}{\,\mathrm{d}}
\renewcommand{\d}{\mathrm{d}}
\newcommand{\R}{\mathbb{R}}
\newcommand{\N}{\mathbb{N}}

\newcommand{\X}{{\bf X}}

\newcommand{\var}{\text{-}\mathsf{var}}
\newcommand{\vertiii}[1]{{\left\vert\kern-0.25ex\left\vert\kern-0.25ex\left\vert #1 \right\vert\kern-0.25ex\right\vert\kern-0.25ex\right\vert}}
\newcommand{\Lip}{\mathrm{Lip}}

\renewcommand{\L}{\mathcal{L}}

\title{On Sobolev rough paths}

\author[Liu]{Chong Liu}
\address{Chong Liu, University of Oxford, United Kingdom}
\email{chong.liu@maths.ox.ac.uk}

\author[Pr\"omel]{David J. Pr\"omel}
\address{David J. Pr\"omel, University of Mannheim, Germany}
\email{proemel@uni-mannheim.de}

\author[Teichmann]{Josef Teichmann}
\address{Josef Teichmann, Eidgen\"ossische Technische Hochschule Z\"urich, Switzerland}
\email{josef.teichmann@math.ethz.ch}

\date{\today}

\begin{document}

\begin{abstract}
  We introduce the space of rough paths with Sobolev regularity and the corresponding concept of controlled Sobolev paths. Based on these notions, we study rough path integration and rough differential equations. As main result, we prove that the solution map associated to differential equations driven by rough paths is a locally Lipschitz continuous map on the Sobolev rough path space for any arbitrary low regularity~$\alpha$ and integrability~$p$ provided $\alpha >1/p$. 
\end{abstract}

\maketitle


% Classification
% --------------
% 60L20 Rough analysis (Rough paths)

%\tableofcontents

\section{Introduction}

Loosely speaking, a \textit{rough path} in the sense of T. Lyons~\cite{Lyons1998} is a path $\X$ from $[0,1]$ taking values in a suitable algebraic structure, namely the step-$N$ free nilpotent group $G^N(\mathbb{R}^d)$, and possessing sufficient regularity such as $\alpha$-H\"older continuity or $1/\alpha$-variation for $\alpha >1/N$. While rough path theory found many successful applications over the past two decades, its original motivation is to study so-called \textit{rough differential equations} (RDEs) 
\begin{equation}\label{eq:intro RDE}
  \d Y_t = V(Y_t)\dd \X_t, \quad Y_0 = y_0, \quad t \in [0,1],
\end{equation}
where $y_0 \in \R^e$ is an initial value and $V$ is a smooth vector field on $\R^e$  mapping into the linear operators from $\R^d$ to $\R^e$.

As long as the driving signal~$\X=X$ is simply a path $X\colon [0,1]\to\R^d$, which is at least weakly differentiable with $p$-integrable derivative, that is, $X$ belongs to the Sobolev space $W^{1}_p$, the rough differential equation~\eqref{eq:intro RDE} is a classical object in analysis known as controlled ordinary differential equations, see, e.g., \cite{Filippov1988}. Controlled differential equations appear in several different areas of analysis or geometry. The regularity of the driving signal plays an important role there, as well as the metric properties of the spaces of driving signals, e.g.~its reflexivity or strict convexity: for instance in sub-Riemannian geometry, see \cite{Montgomery2002}, the geodesic problem leads to a minimal energy problem for horizontal paths, which can be solved due to the Hilbert space structure of the space of driving signals. Or in machine learning, where controlled differential equations can be regarded as a continuous depth version for deep feed forward neural networks, see \cite{Cuchiero2019}, the target problem in the spirit of the Chow-Rashevskii theorem leads to a minimal energy problem involving again strong metric properties of the space of driving signals.

The problem of choosing appropriate spaces of rough paths becomes considerably more involved when path regularity decreases, since non-linear effects appear. As soon as the driving signal~$X$ is a sample path of a stochastic process like a Brownian motion, the RDE~\eqref{eq:intro RDE} (then also known as stochastic differential equation) cannot be treated anymore by classical methods from real analysis, cf. \cite{Lyons1991}, but needs stochastic methods. An alternative way, which fully clarifies the nature of the appearing non-linearity, is to assume that the driving signal~$\X$ is a rough path in the sense of T.~Lyons. Then, the theory of rough paths establishes that~\eqref{eq:intro RDE} possesses a unique solution~$Y$ and the solution map $\X\mapsto Y$ is locally Lipschitz continuous with respect to suitable rough path metrics. In the context of rough path theory the map $\X\mapsto Y$ is often called \textit{It\^o--Lyons map}. For more detailed introductions to rough path theory we refer to~\cite{Lyons2007,Lejay2009,Friz2010,Friz2014}.

It is well-known that the space of rough paths can be introduced in various ways by postulating different regularity properties in the definition of a rough path. Of course, all these rough path spaces share the fundamental feature that the It\^o--Lyons map is locally Lipschitz continuous with respect to the corresponding distances on the underlying rough path spaces, see e.g. \cite{Friz2010} or \cite{Friz2018} and the references therein. The aim of the present paper is to introduce the fractional Sobolev regularity as defining regularity property of a rough path and show the local Lipschitz continuity of the It\^o-Lyons map.

The metric structure of Sobolev spaces $W^\alpha_p$ for $1<p<+\infty$ offers many favourable properties which are not provided by the frequently used distances on the rough path spaces such as H\"older or $p$-variation distances. Among others, let us mention for instance that the real-valued Sobolev spaces are known to be strictly convex, separable, reflexive, UMD Banach spaces of martingale type~$2$. Some of these properties are essential to solve optimization problems or to set up stochastic integration. Furthermore, Sobolev settings allow for better moment estimates in the context of stochastic partial differential equations. In the theory of regularity structure~\cite{Hairer2014} and of paracontrolled distributions~\cite{Gubinelli2015}, which are both closely related to rough path theory, the aforementioned favourable properties lead to a recent effort to introduce Sobolev distances or the even more general Besov distances in these theories, see e.g. \cite{Hairer2017,Liu2016,Hensel2020} for regularity structures and e.g. \cite{Promel2016,Martin2019,Hoshino2020} for paracontrolled distributions.

The fractional Sobolev spaces appear naturally in the study of differential equations in classical analysis and of stochastic differential equations, for example, when working with with Cameron--Martin spaces. These Sobolev spaces appear even in the context of rough differential equations, see \cite{Cass2010}. However, Sobolev distances on the space of rough paths are not used so far. The main reason for this stems from the fact that, in general, it has been unclear whether the It\^o--Lyons map is locally Lipschitz continuous with respect to the inhomogeneous Sobolev distance, without losing regularity\footnote{i.e.~mapping $X$ with Sobolev regularity $\alpha$ to $Y$ with Sobolev regularity $\beta$ for $\beta <\alpha$}.

In Section~\ref{sec: Sobolev rough path} and~\ref{sec: Sobolev controlled rough paths} we introduce the space of Sobolev rough paths and the corresponding space of controlled paths of Sobolev type. As a first step, we demonstrate that these spaces lead to rough path integration with its known properties and to standard stability results as usually offered by rough path theory. Our approach is based on a novel discrete characterization of (non-linear) Sobolev spaces (see~\cite{Liu2020}) in combination with classical estimates from rough path theory and Sobolev-variation embedding theorems (see also~\cite{Friz2006}).

In Section~\ref{sec: Sobolev RDE} and \ref{sec: Ito map} we manage to obtain the local Lipschitz continuity of the It\^o--Lyons map acting on the space of Sobolev rough paths with arbitrary low regularity $\alpha>0$ and integrability $p$ such that $\alpha >1/p$. Although our proof again utilizes some of the sophisticated estimates from rough path theory, the Sobolev distances creates some new challenges mainly because of its missing direct link to a control function. Indeed, let us recall that numerous definitions of rough path spaces rely on metrics closely related the concept of so-called control functions~$\omega$, which provide good estimates of increments of the type $|Y_{t}-Y_{s}|\leq \omega (s,t)$ such as the $p$-variation norm with $\omega(s,t):=\|Y\|_{p\var;[s,t]}^p$. While these type of estimates make it convenient to work $p$-variation or related semi-norms, the fractional Sobolev norm does not come with such convenient estimates of increments of rough paths.

The present work confirms that the Sobolev regularity offers a suitable topology on the space of rough paths and that the solution theory for rough differential equations naturally extends the classical solution theory of controlled ordinary differential equations based on Sobolev spaces. Additionally, this guarantees the access to the above mentioned favourable properties.
 
\medskip
\noindent{\bf Organization of the paper:} In Section~\ref{sec: Sobolev rough path} we introduce the space of Sobolev rough paths. Controlled paths of Sobolev type are discussed in Section~\ref{sec: Sobolev controlled rough paths} and rough differential equations driven by Sobolev rough paths are studied in Section~\ref{sec: Sobolev RDE}. The local Lipschitz continuity of the It\^o--Lyons map acting on the space of Sobolev rough paths with arbitrary low regularity is provided in Section~\ref{sec: Ito map}.

\medskip
\noindent{\bf Acknowledgment:} C.~Liu and J.~Teichmann gratefully acknowledge support by the ETH foundation. D.J.~Pr{\"o}mel and J.~Teichmann gratefully acknowledge support by the SNF Project 163014.

\section{Sobolev rough path space}\label{sec: Sobolev rough path}

The definition of a rough path in the sense T. Lyons \cite{Lyons1998} basically consists of two components: an algebraic structure and an analytic regularity condition. While we work with the standard algebraic structure, we shall introduce a Sobolev regularity, which is in contrast to the common approaches in rough path theory, cf. \cite{Lyons2007,Lejay2009,Friz2010,Friz2014}.

We start by recalling some basic notation and definitions from rough path theory, as used e.g. in \cite{Friz2010}, and introduce the underlying algebra structure, which can be conveniently described by the free nilpotent Lie group $G^N(\R^d)$. Let $\R^d$ be the Euclidean space with norm~$| \cdot |$ for $d\in \N$. The tensor algebra over $\R^d$ is defined by 
\begin{equation*}
  T(\R^d) := \bigoplus_{n=0}^\infty (\R^d)^{\otimes n}
\end{equation*}
where $\big(\mathbb{R}^d\big)^{\otimes n}$ denotes the $n$-tensor space of $\mathbb{R}^d$ with the convention $(\R^d)^{\otimes 0}:=\mathbb{R}$. We equip $T(\R^d)$ with the standard addition $+$, tensor multiplication~$\otimes$ and scalar product.

Let $C^{1\var}([0,1];\mathbb{R}^d)$ be the space of all continuous functions $Z\colon [0,1]\to \R^d$ of finite variation. For $N\in\mathbb{N}$ and a path $Z\in C^{1\var}([0,1];\mathbb{R}^d)$, its step-$N$ signature is defined by 
\begin{align*}
  S_N(Z)_{s,t}:=&\bigg (1, \int_{s<u<t}\dd Z_u, \dots, \int_{s<u_1<\dots <u_N<t}\dd Z_{u_1} \otimes \cdots \otimes \d Z_{u_N} \bigg) \\
  &\in T^N(\mathbb{R}^d):= \bigoplus_{k=0}^N \big(\mathbb{R}^d\big)^{\otimes k}\subset T(\R^d),
\end{align*}
cf. \cite[Definition~7.2]{Friz2010}. The corresponding space of all these lifted paths is the step-$N$ free nilpotent group (w.r.t. $\otimes$) 
\begin{equation*}
  G^N(\mathbb{R}^d) := \{S_N(Z)_{0,1} \,:\, Z\in C^{1\var}([0,1];\mathbb{R}^d)\}\subset T^N(\mathbb{R}^d).
\end{equation*}
On $G^N(\mathbb{R}^d)$ one usually works with two types of complete metrics: the first metric is given by
\begin{equation*}
  \rho(g,h) := \max_{i=1,\dots,N}|\pi_i(g - h)|\quad \text{for} \quad g,h \in G^N(\R^d),
\end{equation*}
where $\pi_i$ denotes the projection from $\bigoplus_{i=0}^N (\R^d)^{\otimes i}$ onto the $i$-th level. We set $|g|:=\rho(g, 1)$ for $g \in G^N(\R^d)$. The second one is the Carnot--Caratheodory metric $d_{cc}$, which is given by
\begin{equation*}
  d_{cc}(g,h) := \|g^{-1} \otimes h\|_{cc}  \quad \text{for} \quad  g,h \in G^N(\R^d),
\end{equation*}
where $\| \cdot \|_{cc} $ is the Carnot--Caratheodory norm defined via \cite[Theorem~7.32]{Friz2010}, cf. \cite[Definition~7.41]{Friz2010}. These two metrics are in general not equivalent (unless $d = 1$) in the sense that there exist constants $C_1,C_2$ such that $C_1 \rho(g,h) \leq d_{cc}(g,h) \leq C_2 \rho(g,h)$ for all $g,h \in G^N(\R^d)$. However, one has  
\begin{equation*}
  \rho(g,h) \leq C d_{cc}(g,h) \quad \text{and} \quad  d_{cc}(g,h) \leq C \rho(g,h)^{1/N},
\end{equation*}
for some constant $C>0$, uniformly on bounded sets (w.r.t. the Carnot--Caratheodory norm), see \cite[Proposition~7.49]{Friz2010}. In the following we equip the free nilpotent Lie group $G^N(\R^d)$ with the Carnot--Caratheodory metric~$d_{cc}$, which turns $G^N(\R^d)$ into a complete geodesic metric space. For a path $\X\colon [0,1]\to G^N(\R^d)$, we set $\X_{s,t} := \X_s^{-1} \otimes \X_t$ for any subinterval $[s,t] \subset [0,1]$. We refer to \cite[Chapter~7]{Friz2010} for a more comprehensive introduction to $G^N(\R^d)$. 
\medskip

Next we introduce the analytic regularity conditions required on a rough path. A partition~$\pi$ of an interval $[s,t]$ is a collection of finitely many essentially disjoint interval covering $[s,t]$, i.e., $\mathcal{P} := \{ [t_{k-1},t_{k}]\,:\,s=t_0<t_1<\dots<t_n=t,\, n\in \N\}$. In this case we write $\mathcal{P}\subset [s,t]$ indicating that $\mathcal{P}$ is a partition of the interval $[s,t]$. Furthermore, for such a partition $\mathcal{P}$ and a function $\chi \colon \{ (u,v) \,:\,  s\leq u <v\leq t\}\to \R$ we use the abbreviation
\begin{equation*}
  \sum_{[u,v]\in \mathcal{P}} \chi(u,v):= \sum_{i=0}^{n-1} \chi(t_i,t_{i+1}).
\end{equation*}
In the following, if not otherwise specified, $(E,d)$ denotes a metric space and $C([0,1];E)$ stands for the set of all continuous functions $f\colon [0,1] \to E$. We can obtain a metric thereon by $d_\infty(f,g) := \sup_{0\leq t\leq 1}d(f(t),g(t))$. If $E$ is normed vector space with norm~$\|\cdot \|$, we define $\|f\|_{\infty}:=\sup_{0\leq t\leq 1}\|f(t)\|$. The \textit{$q$-variation} of a function $f\in C([0,1];E)$ is defined by
\begin{equation}\label{eq:p-varition}
  \|f\|_{q\var;[s,t]}:=\bigg(\sup_{\mathcal{P}\subset [s,t]} \sum_{[u,v]\in \mathcal{P}} 
  d(f_u,f_v)^q  \bigg)^{1/q},\quad q\in[1,+\infty),
\end{equation}
where the supremum is taken over all partitions $\mathcal{P}$ of the interval $[s,t]$. The set of all functions $f\in C([0,1];E)$ with $\|f\|_{q\var}:= \|f\|_{q\var;[0,1]}<\infty$ is denoted by $C^{q\var}([0,1];E)$. For $r \in \R_{+}$ we set 
\begin{equation*}
  [r] := \sup\{n \in \mathbb{Z}\, :\, n \leq r\}\quad \text{and}\quad\lfloor r \rfloor := \sup\{n \in \mathbb{Z}\, :\, n < r\}.
\end{equation*}

The space of all \textit{weakly geometric rough paths} of finite $q$-variation is then given by 
\begin{equation*}
  \Omega^q:=C^{q\var}([0,1];G^{\lfloor q \rfloor}(\mathbb{R}^n))
  := \bigg\{\X\in C([0,1];G^{\lfloor q \rfloor}(\mathbb{R}^n))\,:\, \|\X\|_{q\var} <\infty \bigg\},
\end{equation*}
where $\|\,\cdot\,\|_{q\var}$ is the $q$-variation with respect to the metric space~$(G^{\lfloor q \rfloor}(\mathbb{R}^n),d_{cc})$ as defined in~\eqref{eq:p-varition}. Let us remark that $\|\,\cdot\,\|_{q\var}$ on $\Omega^q$ is often called the homogeneous rough path norm because it is homogeneous with respect to the dilation map on $T^{\lfloor q \rfloor}(\R^n)$, cf.~\cite[Definition~7.13]{Friz2010}. The $q$-variation norm is frequently used in rough path theory but it is well-known that there exists a cascade of good metrics to measure the regularity of a rough path, see e.g. \cite{Friz2018} for a discussion about rough path metrics. 

In contrast to the commonly used metrics in rough path theory, we shall consider fractional Sobolev metrics. For this purpose, let recall the definition of Sobolev regularity for functions mapping into a metric space~$E$. For $\alpha \in (0,1)$, $p \in (1,+\infty)$ and a function $f\in C([0,T];E)$ we define the \textit{fractional Sobolev regularity} by
\begin{equation}\label{eq:Sobolev norm}
  \|f\|_{W^{\alpha}_p;[s,t]}:= \bigg( \iint_{[s,t]^2} \frac{d(f(u),f(v))^p}{|v-u|^{\alpha p+1}} \dd u \dd v\bigg)^{1/p}
\end{equation} 
and in the case of $p=+\infty$ we set 
\begin{equation*}
  \|f\|_{W^{\alpha}_p;[s,t]}:= \sup_{u,v\in[s,t],\,} \frac{d(f(u),f(v))}{|v-u|^\alpha},
\end{equation*} 
where $[s,t]\subset [0,1]$. The latter case is also known as H\"older regularity. Furthermore, we set $\|f\|_{W^{\alpha}_p}:=\|f\|_{W^{\alpha}_p;[0,1]}$. The space $W^{\alpha}_p([0,1];E)$ consists of all continuous functions $f\colon [0,1]\to E$ such that $\|f\|_{W^{\alpha}_p}<\infty$. For a continuous function $f\colon [0,1] \to E$, the fractional Sobolev (semi)-distance can be equivalently defined in a discrete way by 
\begin{equation}\label{eq:discrete Sobolev norm}
  {\|f\|}_{W^{\alpha}_p,(1)} :=\bigg( \sum_{j \geq 0} 2^{j( \alpha p - 1)} \sum_{m=0}^{2^j-1}  d \big(f(\frac{m}{2^j}),f(\frac{m+1}{2^j}) \big)^p \bigg)^{1/p}.
\end{equation}
Indeed, the next theorem recalls this equivalence of the Sobolev distances, which was shown in \cite[Theorem~2.2]{Liu2020} even  for the more general class of Besov spaces. 

\begin{theorem}\label{thm:Sobolev norm equivalent}
  Let $\alpha \in (0,1)$ and $p\in (1,+\infty]$ be such that $\alpha >1/p$. Then, there exist two constants $C_1,C_2>0$ depending only on $\alpha$ and $p$ such that 
  \begin{equation*}
    C_1  {\|f\|}_{W^{\alpha}_p,(1)}  \leq \|f\|_{W^{\alpha}_p} \leq C_2 {\|f\|} _{W^{\alpha}_p,(1)},\quad f\in C([0,1];E).
  \end{equation*}
\end{theorem}

The Sobolev regularity leads naturally to the notion of (fractional) Sobolev rough paths. 

\begin{definition}[Sobolev rough path]\label{def:Sobolev geometric rough path}
  Let $\alpha \in (0,1)$ and $p \in (1,+\infty]$ be such that $\alpha > 1/p$. The space $W^{\alpha}_p([0,1];G^{[\frac{1}{\alpha}]}(\mathbb{R}^d))$ consists of all paths $\X\colon[0,1]\to G^{[\frac{1}{\alpha}]}(\mathbb{R}^d) $ such that 
  \begin{equation*}
    \| \X \|_{W^{\alpha}_p} := \Big(\iint_{[0,1]^2}\frac{d_{cc}(\X_s,\X_t)^p}{|t-s|^{\alpha p + 1}} \dd s \dd t\Big)^{1/p} <+\infty.
  \end{equation*}
  The space $W^{\alpha}_p([0,1];G^{[\frac{1}{\alpha}]}(\mathbb{R}^d))$ is called the \textit{weakly geometric Sobolev rough path space} and $\X\in W^{\alpha}_p([0,1];G^{[\frac{1}{\alpha}]}(\mathbb{R}^d))$ is called a \textit{weakly geometric rough path of Sobolev regularity} $(\alpha,p)$ or short \textit{Sobolev rough path}. 
\end{definition}

\begin{remark}
  Assuming that $\alpha \in (0,1)$ and $p \in (1,+\infty]$ with $\alpha > 1/p$, every weakly geometric rough path of Sobolev regularity $(\alpha,p)$ is also H\"older continuous of order $\alpha-1/p$, which can be seen with the help of the Garsia--Rodemich--Rumsey inequality, see e.g.~\cite[Theorem~A.1]{Friz2010}.
  
  Furthermore, the metric structure provided by the Sobolev metric allows to conveniently approximate Sobolev rough path by geodesic interpolations along the dyadic numbers, see \cite[Section~3.3]{Liu2018}.
\end{remark}

In order to obtain the Lipschitz continuity of the solution map associated to differential equations driven by Sobolev rough paths, we need to introduce an inhomogeneous Sobolev distance $\hat{\rho}_{W^\alpha_p}$ on $W^\alpha_p([0,1];G^{[\frac{1}{\alpha}]}(\R^d))$ defined by 
\begin{equation}\label{eq:inhomogeneous Sobolev distance}
  \hat{\rho}_{W^\alpha_p} (\X^1,\X^2):= \sum_{k =1}^{[\frac{1}{\alpha}]} \hat{\rho}^{(k)}_{W^\alpha_p}(\X^1,\X^2),
  \quad \text{for}\quad \X^1,\X^2 \in W^\alpha_p([0,1];G^{[\frac{1}{\alpha}]}(\R^d)),
\end{equation}
where 
\begin{equation*}
  \hat{\rho}^{(k)}_{W^\alpha_p}(\X^1,\X^2) := \Big(\sum_{j \ge 0} 2^{j(\alpha p -1)}\sum_{i=1}^{2^j} |\pi_k(\X^1_{(i-1)2^{-j},i2^{-j}} -   \X^2_{(i-1)2^{-j},i2^{-j}})|^{\frac{p}{k}}\Big)^{\frac{k}{p}}.
\end{equation*}
Note that $\hat{\rho}_{W^\alpha_p}$ is the inhomogeneous counterpart of the discretely defined homogeneous Sobolev norm~\eqref{eq:discrete Sobolev norm}, which is equivalent to the (classical) Sobolev metric as defined in~\eqref{eq:Sobolev norm}, see Theorem~\ref{thm:Sobolev norm equivalent}. Therefore, by using the equivalence of the homogeneous norms on the Carnot group~$G^{[\frac{1}{\alpha}]}(\R^d)$ (see \cite[Theorem~7.44]{Friz2010}) and the triangle inequality, one can verify that $\hat{\rho}_{W^\alpha_p}(\X^1,\X^2) \lesssim \|\mathbf{X}^1\|_{W^\alpha_p,(1)} + \|\mathbf{X}^2\|_{W^\alpha_p,(1)}<+ \infty$ for $\X^1,\X^2$ in $W^\alpha_p([0,1];G^{[\frac{1}{\alpha}]}(\R^d))$, where the proportional constant may depend on $\|\mathbf{X}^1\|_{W^\alpha_p}$ and $\|\mathbf{X}^2\|_{W^\alpha_p}$.

\begin{remark}
  Note that there is also a canonical way to introduce the inhomogeneous Sobolev distance analogously to the integral definition of the homogeneous Sobolev norm~\eqref{eq:Sobolev norm}, which is expected to be equivalent to the discretely defined inhomogeneous Sobolev distance~\eqref{eq:inhomogeneous Sobolev distance}. However, already in the case of homogeneous Sobolev norms, it was a challenging task to show the equivalence of the Sobolev norm via integrals~\eqref{eq:Sobolev norm} and the discretely defined Sobolev norm~\eqref{eq:discrete Sobolev norm}, see \cite{Liu2020}. 
\end{remark}

Moreover, let us recall the inhomogeneous mixed H\"older-variation distance as introduced in~\cite[Section~3.2]{Friz2018}, which is given by 
\begin{equation*}
  \rho_{\tilde V^{\alpha,p}}(\X^1,\X^2):= \max_{k=1,\dots,N}\rho^{(k)}_{\tilde V^{\alpha,p};[0,1]}(\X^1,\X^2),
\end{equation*}
where
\begin{equation*}
  \rho^{(k)}_{\tilde V^{\alpha,p};[s,t]}(\X^1,\X^2):= \sup_{\mathcal{P}\subset [s,t]} \bigg (\sum_{[u,v] \in \mathcal{P}} \frac{\rho^{(k)}_{1/\alpha\var;[u,v]}(\X^1,\X^2)^{\frac{p}{k}}}{|u-v|^{\alpha p-1}} \bigg)^{\frac{k}{p}}, \quad [s,t]\subset [0,1].
\end{equation*}
By \cite[Theorem~2.11 and Corollary~2.12]{Friz2018} and the equivalence of homogeneous norms on the Carnot group~$G^{[\frac{1}{\alpha}]}(\R^d)$, one immediately has that $\rho_{\tilde{V}^\alpha_p}(\X^1,\X^2) \lesssim \|\mathbf{X}^1\|_{W^\alpha_p,(1)} + \|\mathbf{X}^2\|_{W^\alpha_p,(1)}  <+ \infty$ for $\X^1,\X^2$ in $W^\alpha_p([0,1];G^{[\frac{1}{\alpha}]}(\R^d))$, as we established the same bound for $\hat{\rho}_{W^\alpha_p}(\X^1,\X^2)$.\

\medskip
In the following we frequently use the abbreviations: For two real functions $a,b$ depending on variables~$x$ we write $a\lesssim b$ or $a\lesssim_z b$ if there exists a constant $C(z)>0$ such that $a(x) \leq C(z)\cdot b(x)$ for all $x$, and $a\sim b$ if $a\lesssim b$ and $b\lesssim a$ hold simultaneously.

\section{On controlled paths of Sobolev type}\label{sec: Sobolev controlled rough paths}

After having introduced the space of Sobolev rough paths, one wants to ensure that Sobolev rough paths lead to a fully fledged rough path integration and allow to set up a solution theory for rough differential equations. In order to demonstrate the difficulties arising by working with Sobolev rough paths, let us consider first the two level case $W^\alpha_p([0,1];G^2(\R^d))$ for $\alpha \in (1/3,1/2)$ and $p \in (1,+\infty]$ with $\alpha >1/p$. We shall deal with the general case in Section~\ref{sec: Sobolev RDE} and~\ref{sec: Ito map}. In this section we follow the approach using controlled paths as introduced by M. Gubinelli~\cite{Gubinelli2004} and the sewing lemma (see \cite[2.1~ Lemma]{Feyel2006}) to develop rough path integration. See also see textbook~\cite{Friz2014} for this approach.
\medskip

For this purpose we recall the notion of controlled paths possessing $1/\alpha$-variation regularity. Let $\mathcal{L}(\mathbb{R}^n;\mathbb{R}^m)$ be the space of linear operators from $\mathbb{R}^n$ to $\mathbb{R}^m$, $\Delta := \{(s,t) \in [0,1]^2: s<t \}$ and let $B$ be a Banach space with norm $|\cdot|$. For $q\in [1,+\infty)$ and for a continuous function $F\colon \Delta \to B$ we define 
\begin{equation*}  
  \|F\|_{q\var;[s,t]} := \bigg( \sup_{\mathcal{P}\subset[s,t]} \sum_{[u,v] \in \mathcal{P}}|F_{u,v}|^q\bigg)^{1/q},
\end{equation*}
for $[s,t]\subset[0,1]$, and $\|F\|_{q\var}:=\|F\|_{q\var;[0,1]}$. A pair $(Y,Y^\prime)$ is a \textit{controlled path} with respect to a given rough path $\X \in C^{\frac{1}{\alpha}\var}([0,1];G^2(\R^d))$ if 
\begin{itemize}
  \item[(i)] $(Y,Y^\prime)\in C^{\frac{1}{\alpha}\var}([0,1];\mathcal{L}(\R^d,\R^e) \oplus \mathcal{L}(\R^d \otimes \R^d; \R^e))$ and
  \item[(ii)] $R^Y\colon \Delta \to \R^e$, given by $R^Y_{s,t} := Y_{s,t} - Y^\prime_s X_{s,t}$ for $(s,t)\in\Delta$, satisfies $\|R^Y\|_{\frac{1}{2\alpha}\var} < +\infty$.
\end{itemize}
The corresponding space of all such controlled paths with respect to $\X \in C^{\frac{1}{\alpha}\var}([0,1];G^2(\R^d))$ is denoted by $\mathcal{D}_{\X}^{1/\alpha\var}([0,T];\R^e)$.

By the Sobolev-variation embedding theorem \cite[Theorem~2]{Friz2006} the Sobolev rough path space $W^\alpha_p([0,1];G^2(\R^d))$ space can be embedded into the space $C^{\frac{1}{\alpha}\var}([0,1];G^2(\R^d))$. In particular, this implies that for any controlled path $(Y,Y^\prime)\in \mathcal{D}_{\X}^{1/\alpha\var}([0,T];\R^d)$ the standard rough path integral $\int Y \dd \X$ exists, cf. \cite[Theorem~4.9]{Perkowski2016}, \cite[Theorem~1]{Gubinelli2004} and \cite[Theorem~4.10]{Friz2014}, and $\int Y \dd \X$ possesses the same $1/\alpha$-variation as the rough path $\X$. In the next lemma, we make a first observation how these statements transfer into the Sobolev setting.

\begin{lemma}\label{lemma: rough integral of Sobolev rough path lose regularity}
  Let $\X=(X,\mathbb{X})$ be a Sobolev rough path in $W^\alpha_p([0,1];G^2(\R^d))$ for $\alpha \in (1/3,1/2)$ and $p \in (1,+\infty)$ with $\alpha > 1/p$. Let $(Y,Y^\prime)\in \mathcal{D}_{\X}^{1/\alpha\var}([0,T];\R^e)$ be an $\R^e$-valued controlled rough path. Then, the rough path integral $\int Y \dd \X$ exists and belongs to the Sobolev space $W_p^{\alpha^\prime}([0,1];\R^e)$ for every $\alpha^\prime < \alpha$.
\end{lemma}

Before proving Lemma~\ref{lemma: rough integral of Sobolev rough path lose regularity}, let us recall the notion of control functions: A function $\omega\colon\Delta\to [0,+\infty)$ is called control function if $\omega(s,s)=0$ for $s\in[0,1]$ and $\omega$ is super-additive.

\begin{proof}
  By \cite[Theorem~4.9]{Perkowski2016}, one has
  \begin{align}\label{eq:Young estimate}
    \begin{split}
    &\Big| \int_s^t Y_r \dd \X_r - Y_s X_{s,t} - Y^\prime_s\mathbb{X}_{s,t}\Big| \\
    &\qquad\lesssim \|R^Y\|_{\frac{1}{2\alpha}\var;[s,t]}\|X\|_{\frac{1}{\alpha}\var;[s,t]}+\|Y^\prime\|_{\frac{1}{\alpha}\var;[s,t]}\|\mathbb{X}\|_{\frac{1}{2\alpha}\var;[s,t]}.
    \end{split}
  \end{align}
  Now we fix an $\alpha^\prime < \alpha$. Thanks to the discrete characterization of Sobolev rough path (Theorem~\ref{thm:Sobolev norm equivalent}, in order to show that $\int Y_r \dd \X_r \in W^{\alpha^\prime}_p([0,1];\R^e)$ it suffices to prove that
  \begin{equation*}
    \Big\|\int Y_r \dd \X_r\Big\|_{W^{\alpha^\prime}_p,(1)}^p := \sum_{j=0}^\infty \sum_{i=1}^{2^j} 2^{j(\alpha^\prime p - 1)}\Big|\int_{\frac{i-1}{2^j}}^{\frac{i}{2^j}} Y_r \dd \X_r\Big|^p < +\infty.
  \end{equation*}
  Indeed, applying \eqref{eq:Young estimate}, we get 
  \begin{align*}
    \Big\|\int Y_r \dd \X_r\Big\|_{W^{\alpha^\prime}_p,(1)}^p
    \lesssim &\sum_{j=0}^\infty \sum_{i=1}^{2^j} 2^{j(\alpha^\prime p - 1)}\Big|\int_{\frac{i-1}{2^j}}^{\frac{i}{2^j}} Y_r \dd \X_r - Y_{\frac{i-1}{2^j}}X_{\frac{i-1}{2^j},\frac{i}{2^j}} - Y^\prime_{\frac{i-1}{2^j}}\mathbb{X}_{\frac{i-1}{2^j},\frac{i}{2^j}} \Big|^p \\
    &+ \sum_{j=0}^\infty \sum_{i=1}^{2^j} 2^{j(\alpha^\prime p - 1)}\Big|Y_{\frac{i-1}{2^j}}X_{\frac{i-1}{2^j},\frac{i}{2^j}} + Y^\prime_{\frac{i-1}{2^j}}\mathbb{X}_{\frac{i-1}{2^j},\frac{i}{2^j}}\Big|^p  \\
    \lesssim &\sum_{j=0}^\infty \sum_{i=1}^{2^j} 2^{j(\alpha^\prime p - 1)}\|R^Y\|_{\frac{1}{2\alpha}\var;[\frac{i-1}{2^j},\frac{i}{2^j}]}^p\|X\|_{\frac{1}{\alpha}\var;[\frac{i-1}{2^j},\frac{i}{2^j}]}^p \\
    &+ \sum_{j=0}^\infty \sum_{i=1}^{2^j} 2^{j(\alpha^\prime p - 1)}\|Y^\prime\|_{\frac{1}{\alpha}\var;[\frac{i-1}{2^j},\frac{i}{2^j}]}^{p}\|\mathbb{X}\|_{\frac{1}{2\alpha}\var;[\frac{i-1}{2^j},\frac{i}{2^j}]}^p\\
    &+ \sum_{j=0}^\infty \sum_{i=1}^{2^j} 2^{j(\alpha^\prime p - 1)}\Big|Y_{\frac{i-1}{2^j}}X_{\frac{i-1}{2^j},\frac{i}{2^j}} + Y^\prime_{\frac{i-1}{2^j}}\mathbb{X}_{\frac{i-1}{2^j},\frac{i}{2^j}}\Big|^p.
  \end{align*}
  Now we estimate separately each of the terms of the above sum.
  
  For the last term, since $\X \in W^\alpha_p([0,1];G^2(\R^d))$, we immediately have 
  \begin{equation*}
    \sum_{j=0}^\infty \sum_{i=1}^{2^j} 2^{j(\alpha^\prime p - 1)}\Big|Y_{\frac{i-1}{2^j}}X_{\frac{i-1}{2^j},\frac{i}{2^j}} + Y^\prime_{\frac{i-1}{2^j}}\mathbb{X}_{\frac{i-1}{2^j},\frac{i}{2^j}}\Big|^p  \lesssim \big(\|Y\|_{\infty} + \|Y^\prime\|_{\infty}\big)^p\|\X\|_{W^\alpha_p}^p < +\infty.
  \end{equation*}
  
  For the second term, by \cite[Proposition~4.3]{Liu2020} we observe that 
  \begin{equation*}
    \|\X\|_{\frac{1}{\alpha}\var;[s,t]}^{\frac{1}{\alpha}} \lesssim \|\X\|_{W^\alpha_p;[s,t]}^{\frac{1}{\alpha}}|t-s|^{1-\frac{1}{\alpha p}}
  \end{equation*}
  for all $s<t$, which implies that 
  \begin{equation*}
    \|\mathbb{X}\|_{\frac{1}{2\alpha}\var;[\frac{i-1}{2^j},\frac{i}{2^j}]}^p \lesssim \|\X\|_{\frac{1}{\alpha}\var;[\frac{i-1}{2^j},\frac{i}{2^j}]} ^{2p} \lesssim \|\X\|_{W^\alpha_p; [\frac{i-1}{2^j},\frac{i}{2^j}]}^{2p} (2^{-j})^{2(\alpha p -1)},
  \end{equation*}
  and consequently that 
  \begin{equation*}
    \|Y^\prime\|_{\frac{1}{\alpha}\var;[\frac{i-1}{2^j},\frac{i}{2^j}]}^{p}\|\mathbb{X}\|_{\frac{1}{2\alpha}\var;[\frac{i-1}{2^j},\frac{i}{2^j}]}^p \lesssim \|Y^\prime\|_{\frac{1}{\alpha}\var;[0,1]}^{p}\|\X\|_{W^\alpha_p; [\frac{i-1}{2^j},\frac{i}{2^j}]}^{2p} (2^{-j})^{2(\alpha p -1)}.
  \end{equation*}
  Since $\|\X\|_{W^\alpha_p; [\cdot,\cdot]}^{2p}$ is a control function, it follows that
  \begin{equation*}
    \sum_{j=0}^\infty \sum_{i=1}^{2^j} 2^{j(\alpha^\prime p - 1)}\|Y^\prime\|_{\frac{1}{\alpha}\var;[\frac{i-1}{2^j},\frac{i}{2^j}]}^{p}\|\mathbb{X}\|_{\frac{1}{2\alpha}\var;[\frac{i-1}{2^j},\frac{i}{2^j}]}^p 
    \lesssim \sum_{j=0}^\infty 2^{-j(\alpha p -1)} \|Y^\prime\|_{\frac{1}{\alpha}\var;[0,1]}^{p}\|\X\|_{W^\alpha_p; [0,1]}^{2p};
  \end{equation*}
  and since $\alpha p - 1 >0$, the sum on the right hand side converges.
  
  For the third term the same reasoning leads to
  \begin{align}\label{eq:estimate remainder}
    &\sum_{j=0}^\infty \sum_{i=1}^{2^j} 2^{j(\alpha^\prime p - 1)}\|R^Y\|_{\frac{1}{2\alpha}\var;[\frac{i-1}{2^j},\frac{i}{2^j}]}^p\|X\|_{\frac{1}{\alpha}\var;[\frac{i-1}{2^j},\frac{i}{2^j}]}^p\\
    &\qquad\lesssim \sum_{j=0}^\infty 2^{j(\alpha^\prime -\alpha)p } \|R^Y\|_{\frac{1}{2\alpha}\var;[0,1]}^p\|\X\|_{W^\alpha_p; [0,1]}^{p}.\nonumber
  \end{align}
  Thanks to the assumption that $\alpha^\prime < \alpha$, the sum on the right hand side of the above inequality converges. Hence, the proof is completed.
\end{proof}

From Lemma~\ref{lemma: rough integral of Sobolev rough path lose regularity} we see that, without adapting the regularity of the controlled path $(Y,Y^\prime)$, one can only guarantee that $\int Y \dd \X$ belongs to the Sobolev space $W^{\alpha^\prime}_p([0,1];\R^e)$ for every $\alpha^\prime < \alpha$. In words, the rough path integral has less regularity than the rough path $\X$. This observation motivates us to introduce a Sobolev topology also on the space of controlled paths. \medskip

Looking again at the third term~\eqref{eq:estimate remainder} in the proof of Lemma~\ref{lemma: rough integral of Sobolev rough path lose regularity}, we notice that to ensure that $\int Y \dd \X$ belongs to $W^\alpha_p([0,1];\R^e)$ separately, one has to find conditions on $R^Y$ such that the series
\begin{equation*}
  \sum_{j=0}^\infty \sum_{i=1}^{2^j} 2^{j(\alpha p - 1)}\|R^Y\|_{\frac{1}{2\alpha}\var;[\frac{i-1}{2^j},\frac{i}{2^j}]}^p\|X\|_{\frac{1}{\alpha}\var;[\frac{i-1}{2^j},\frac{i}{2^j}]}^p
\end{equation*}
converges. Applying the estimates 
\begin{equation*}
  \|X\|_{\frac{1}{\alpha}\var;[\frac{i-1}{2^j},\frac{i}{2^j}]}^p \lesssim \|\X\|_{W^\alpha_p; [\frac{i-1}{2^j},\frac{i}{2^j}]}^{p}2^{-j(\alpha p -1)}
\end{equation*}
to the above series, we essentially need the following condition:
$$
  \sum_{j=0}^\infty \sum_{i=1}^{2^j} \|R^Y\|_{\frac{1}{2\alpha}\var;[\frac{i-1}{2^j},\frac{i}{2^j}]}^p\|\X\|_{W^\alpha_p; [\frac{i-1}{2^j},\frac{i}{2^j}]}^{p} < +\infty.
$$
More explicitly, we need that $\|R^Y\|_{\frac{1}{2\alpha}\var;[\frac{i-1}{2^j},\frac{i}{2^j}]}^p$ can be compared to $2^{-j \beta}$ for some $\beta > 0$ uniformly over all $i=1,\ldots, 2^j$ and $j \ge 1$. This consideration naturally leads us to invoke the so-called mixed H\"older-variation space introduced in \cite{Friz2018}: we shall require that $R^Y$ satisfies that
\begin{equation}\label{eq: mixed Hoelder-variation condition for two variables functions}
  \sup_{\mathcal{P}} \sum_{[u,v] \in \mathcal{P}}\frac{\Big\|R^Y\Big\|_{\frac{1}{2\alpha}\var;[u,v]}^{\frac{p}{2}}}{|u - v|^{\alpha p - 1}} < + \infty.
\end{equation}
Once this is the case, then it follows immediately that 
\begin{equation*}
  \|R^Y\|_{\frac{1}{2\alpha}\var;[\frac{i-1}{2^j},\frac{i}{2^j}]}^p \le C 2^{-2j(\alpha p -1)}
\end{equation*}
for all $i$ and $j$ with $C$ denoting the supremum in \eqref{eq: mixed Hoelder-variation condition for two variables functions}; and then as $\alpha p - 1 > 0$ it holds that
\begin{align*}
  \sum_{j=0}^\infty \sum_{i=1}^{2^j} \|R^Y\|_{\frac{1}{2\alpha}\var;[\frac{i-1}{2^j},\frac{i}{2^j}]}^p\|\X\|_{W^\alpha_p; [\frac{i-1}{2^j},\frac{i}{2^j}]}^{p} &\lesssim \sum_{j=0}^\infty 2^{-2j(\alpha p -1)} \sum_{i=1}^{2^j}\|\X\|_{W^\alpha_p; [\frac{i-1}{2^j},\frac{i}{2^j}]}^{p}\\
  &\le \Big(\sum_{j=0}^\infty 2^{-2j(\alpha p -1)}\Big)\|\X\|_{W^\alpha_p; [0,1]}^{p} < +\infty,
\end{align*}
as wished.

Inspired by the above observations, we introduce the following function space: Let $(B, \|\cdot \|)$ be a Banach space. For $\beta \in (0,1)$ and $q \ge 1$ we use $\tilde{V}^\beta_q(\Delta; B)$ to denote the space of all continuous functions $f \in C(\Delta;B)$ such that
$$
  \sup_{\mathcal{P}} \sum_{[u,v] \in \mathcal{P}}\frac{\|f\|_{\frac{1}{\beta}\var;[u,v]}^{q}}{|u - v|^{\beta q - 1}} < +\infty.
$$
Moreover, for $[s,t]\subset [0,1]$ we define 
$$
  \|f\|_{\tilde{V}^\beta_q;[s,t]} := \bigg(\sup_{\mathcal{P}|_{[s,t]}} \sum_{[u,v] \in \mathcal{P}}\frac{\|f\|_{1/\beta\var;[u,v]}^{q}}{|u - v|^{\beta q - 1}}\bigg)^{\frac{1}{q}}
$$
and $\|f\|_{\tilde{V}^\beta_q}:=\|f\|_{\tilde{V}^\beta_q;[0,1]}$. Let us remark that, if the remainder term $R^Y$ attached to a controlled rough path $(Y,Y^\prime)$ satisfies additionally that $R^Y \in \tilde{V}^{2\alpha}_{\frac{p}{2}}(\Delta; E)$, then the rough integral $\int Y\dd \X$ is an element in $W^\alpha_p([0,1];\R^e)$, by the previous discussion.\medskip

Furthermore, if we want to apply the Banach fixed point theorem to obtain existence and uniqueness results for rough differential equations driven by Sobolev signals~$\X$ within the Sobolev framework, the Sobolev regularity of controlled paths is necessary, i.e., $(Y,Y^\prime)$ should be an element in $W^\alpha_p([0,1];\mathcal{L}(\R^d, \R^e)) \times W^\alpha_p([0,1]; \mathcal{L}(\R^d \otimes \R^d, \R^e))$. In particular, since $Y_{s,t} =  Y^\prime_s X_{s,t}+R^Y_{s,t}$, from the discrete characterization of Sobolev norms~\eqref{eq:discrete Sobolev norm} we see that in this case $R^Y$ satisfies 
$$
  \sum_{j=0}^\infty \sum_{i=1}^{2^j} 2^{j(\alpha p -1)}\Big|R^Y_{\frac{i-1}{2^j},\frac{i}{2^j}}\Big|^p < +\infty.
$$
Hence, let us denote by $\hat{W}^\beta_q(\Delta;\R^n)$ the space of all continuous functions $f \in C(\Delta; \R^n)$ such that
$$
  \|f\|_{\hat{W}^\beta_q} := \Big(\sum_{j=0}^\infty \sum_{i=1}^{2^j} 2^{j(\beta q -1)}\Big|f_{\frac{i-1}{2^j},\frac{i}{2^j}}\Big|^q\Big)^{\frac{1}{q}}<+\infty. 
$$
Hence, in the Sobolev setting the natural definition of controlled paths goes as follows.

\begin{definition}\label{def: Sobolev controlled rough paths}
  Let $\X$ be an element in $W^\alpha_p([0,1];G^2(\R^d))$. A pair $(Y,Y^\prime)$ is called an \textit{controlled path of Sobolev type} $(\alpha,p)$ if $Y \in W^\alpha_p([0,1];\R^n)$, $Y^\prime \in W^\alpha_p([0,1];\mathcal{L}(\R^d;\R^n))$ and $R^Y_{s,t} := Y_{s,t} - Y^\prime_sX_{s,t}$ satisfies that $R^Y \in \tilde{V}^{2\alpha}_{\frac{p}{2}}(\Delta;\R^n) \cap \hat{W}^{2\alpha}_{\frac{p}{2}}(\Delta; \R^n)$. The space of all such controlled rough paths is denoted by $\mathcal{D}^{\alpha,p}_{\X}([0,1];\R^n)$, which is equipped with the norm
  \begin{equation*}%\label{eq: norm of Sobolev controlled rough paths}
    \|(Y,Y^\prime)\|_{\mathcal{D}^{\alpha,p}_{\X}} := \|Y^\prime\|_{W^\alpha_p} + \|R^Y\|_{\tilde{V}^{2\alpha}_{\frac{p}{2}}} + \|R^Y\|_{\hat{W}^{2\alpha}_{\frac{p}{2}}} + |Y_0| + |Y^\prime_0|.
  \end{equation*}
\end{definition}

\begin{remark}%\label{remark: comments on Sobolev controlled rough paths}
  From the definition of $\tilde{V}^{2\alpha}_{\frac{p}{2}}(\Delta;\R^n)$ we can immediately see that if $R^Y \in \tilde{V}^{2\alpha}_{\frac{p}{2}}(\Delta;\R^n)$, then it also has finite $1/2\alpha$-variation. Hence, applying Sobolev-variation embedding results (see \cite[Theorem~2]{Friz2006}) to $(Y,Y^\prime)$, it follows that every controlled path of Sobolev type $(\alpha,p)$ is a controlled path with finite $1/\alpha$-variation. Moreover, using the discrete characterization of Sobolev norms, we can also see that $\|Y\|_{W^\alpha_p}$ can be estimated by $\|R^Y\|_{\hat{W}^{2\alpha}_{\frac{p}{2}}} +\|Y^\prime\|_{W^\alpha_p} + |Y^\prime_0| + \|X\|_{W^\alpha_p}$. Finally, we remark that $(\mathcal{D}^{\alpha,p}_{\X}([0,1];\R^n), \| \cdot\|_{\mathcal{D}^{\alpha,p}_{\X}})$ is a Banach space.
\end{remark}

With the notion of Sobolev rough paths and controlled paths of Sobolev type, one can recover many stability properties known for controlled paths with finite $q$-variations (e.g. under rough path integration, compositions of smooth functions, ...) also for controlled paths of Sobolev type. Let us just mention some of them here.

\begin{lemma}\label{lemma: stability result of rough integrals in Sobolev setting}  
  Let $\X$ be a Sobolev rough path in $W^\alpha_p([0,1];G^2(\R^d))$, $(Y,Y^\prime) \in \mathcal{D}^{\alpha,p}_{\X}([0,1];\R^n)$ be an controlled path of Sobolev type. Let $I_{\X}(Y) := \int Y\dd\X$ be the rough path integral obtained as in Lemma~\ref{lemma: rough integral of Sobolev rough path lose regularity}. Then, one has:
  \begin{enumerate}
	\item[(i)] $(I_{\X}(Y), Y)$ belongs to $\mathcal{D}^{\alpha,p}_{\X}([0,1];\R^n)$.
	\item[(ii)] If $\tilde{\X}$ is another rough path in $W^\alpha_p([0,1];G^2(\R^d))$ and $(\tilde{Y},\tilde{Y}^\prime) \in \mathcal{D}^{\alpha,p}_{\tilde{\X}}([0,1];\R^n)$, then we have the following locally uniform estimates
	\begin{align*}
	  &\| R^{I_{\X}(Y)} - R^{I_{\tilde{\X}}(\tilde{Y})}\|_{\tilde{V}^{2\alpha}_{\frac{p}{2}}}+ \|R^{I_{\X}(Y)} - R^{I_{\tilde{\X}}(\tilde{Y})}\|_{\hat{W}^{2\alpha}_{\frac{p}{2}}}\\
	  &\quad\lesssim \|R^Y - R^{\tilde{Y}}\|_{\tilde{V}^{2\alpha}_{\frac{p}{2}}} + \|R^Y - R^{\tilde{Y}}\|_{\hat{W}^{2\alpha}_{\frac{p}{2}}}  
	  \quad + \|Y^\prime - \tilde{Y}^\prime\|_{W^\alpha_p} + \rho_{\tilde{V}^\alpha_p}(\X, \tilde{\X}) + \hat{\rho}_{W^\alpha_p}(\X, \tilde{\X}),
	\end{align*}
	where $R^{I_{\X}(Y)}$ and $R^{I_{\tilde{\X}}(\tilde{Y})}$ are the remainder terms of $(I_{\X}(Y), Y)$ and $(I_{\tilde{\X}}(\tilde{Y}), \tilde{Y})$, respectively.
  \end{enumerate}
\end{lemma}

\begin{proof}
  (i) We have already shown that with $(Y,Y^\prime) \in \mathcal{D}^{\alpha,p}_{\X}([0,1];\R^n)$, the rough path integral $I_{\X}(Y)$ is well-defined and belongs to $W^\alpha_p([0,1];\R^n)$. Hence, to show the item (i), it only remains to check that the remainder term $R^{I_{\X}(Y)} := \int_s^t Y \dd \X - Y_s X_{s,t}$ belongs to $\tilde{V}^{2\alpha}_{\frac{p}{2}}(\Delta;\R^n) \cap \hat W^{2\alpha}_{\frac{p}{2}}(\Delta;\R^n)$. 
  By \cite[Theorem~4.9]{Perkowski2016}, we note again that
  \begin{align*}
    \Big| \int_s^t Y \d \X - Y_sX_{s,t} \Big| \lesssim |Y^\prime_s\mathbb{X}_{s,t}| + \|R^Y\|_{\frac{1}{2\alpha}\var;[s,t]}\|X\|_{\frac{1}{\alpha}\var;[s,t]} + \|Y^\prime\|_{\frac{1}{\alpha}\var;[s,t]}\|\mathbb{X}\|_{\frac{1}{2\alpha}\var;[s,t]}.
  \end{align*}
  Since $d_{cc}(\X_s,\X_t) \sim |X_{s,t}| + |\mathbb{X}_{s,t}|^{1/2}$, for each $u<v$ in $[0,1]$ we have
  \begin{equation*}
    \|\mathbb{X}\|_{\frac{1}{2\alpha}\var;[u,v]}^{\frac{p}{2}} \lesssim \|\X\|_{\frac{1}{\alpha}\var;[u,v]}^p \lesssim \|\X\|_{W^\alpha_p;[u,v]}^p|u-v|^{\alpha p -1},
  \end{equation*}
  where the last inequality follows again from \cite[Proposition~4.3]{Liu2020}. Then, as $\|\X\|_{W^\alpha_p;[u,v]}^p$ is superadditive in $[u,v]$, we can deduce that $\|\mathbb{X}\|_{\tilde{V}^{2\alpha}_{\frac{p}{2}}} \lesssim \|\X\|_{W^\alpha_p}$. This estimates guarantees that $R^1_{s,t} := |Y^\prime_s\mathbb{X}_{s,t}|$ and $R^3_{s,t}:= \|Y^\prime\|_{\frac{1}{\alpha}\var;[s,t]}\|\mathbb{X}\|_{\frac{1}{2\alpha}\var;[s,t]}$ belong to $\tilde{V}^{2\alpha}_{\frac{p}{2}}(\Delta;\R^n)$. Finally, since we have assumed that $(Y,Y^\prime) \in \mathcal{D}^{\alpha,p}_{\X}([0,1];\R^n)$, it holds that $R^Y \in \tilde{V}^{2\alpha}_{\frac{p}{2}}(\Delta;\R^n)$ by definition and so is $R^2_{s,t} := \|R^Y\|_{\frac{1}{2\alpha}\var;[s,t]}\|X\|_{\frac{1}{\alpha}\var;[s,t]}$. As a consequence, we can conclude that $R^{I_{\X}(Y)} := \int_s^t Y \dd \X - Y_sX_{s,t}$ belongs to $\tilde{V}^{2\alpha}_{\frac{p}{2}}(\Delta;\R^n)$. The fact that $R^{I_{\X}(Y)} \in W^{2\alpha}_{\frac{p}{2}}(\Delta;\R^n)$ can be easily established by following the proof of the item (ii) below.
  
  (ii) Now we bound the term $\|R^{I_{\X}(Y)} - R^{I_{\tilde{\X}}(\tilde{Y})}\|_{\hat{W}^{2\alpha}_{\frac{p}{2}}}$. In the first step above we have seen that $R^{I_{\X}(Y)}_{s,t} = Y^\prime_s\mathbb{X}_{s,t} + h^Y_{s,t}$ with the residue function $h^Y_{s,t}$ having finite $1/3\alpha$ variation. Similarly $R^{I_{\tilde{\X}}(\tilde{Y})}_{s,t} = \tilde{Y}^\prime_s\tilde{\mathbb{X}}_{s,t} + h^{\tilde{Y}}_{s,t}$ for some $h^{\tilde{Y}}$ of finite $1/3\alpha$ variation. Moreover, from the classical sewing lemma (cf. \cite{Friz2014}) we also know that
  \begin{equation*}
    \delta h^Y_{s,u,t}:=h^Y_{s,t} - h^Y_{s,u} - h^Y_{u,t} = -R^Y_{s,u}X_{u,t} - Y^\prime_{s,u}\mathbb{X}_{u,t},
  \end{equation*}
  and the similar relation holds for $\delta h^{\tilde{Y}}_{s,u,t}$ for $s<u<t$. Then, since $3\alpha>1$, the sewing lemma applied to the difference $\delta h^Y_{s,u,t} - \delta h^{\tilde{Y}}_{s,u,t}$ leads to the bound
  \begin{align*}
    |h^Y_{s,t} - h^{\tilde{Y}}_{s,t}| &\lesssim \|R^Y - R^{\tilde{Y}}\|_{\frac{1}{2\alpha}\var;[s,t]}\|X\|_{\frac{1}{\alpha}\var;[s,t]} + \|R^{\tilde{Y}}\|_{\frac{1}{2\alpha}\var;[s,t]}|X_{s,t} - \tilde{X} _{s,t}| \\
    &\qquad+ \|\tilde{Y}^\prime\|_{\frac{1}{\alpha}\var;[s,t]}\|\mathbb{X} - \tilde{\mathbb{X}}\|_{\frac{1}{2\alpha}\var;[s,t]} + \|Y^\prime - \tilde{Y}^\prime\|_{\frac{1}{\alpha}\var;[s,t]}\|\mathbb{X}\|_{\frac{1}{2\alpha}\var;[s,t]}.
  \end{align*}
  Now, inserting $s = \frac{i-1}{2^j}$ and $t=\frac{i}{2^j}$, we can follow the same lines as in the proof of Lemma~\ref{lemma: rough integral of Sobolev rough path lose regularity} to deduce that
  \begin{align*}
    &\sum_{j=0}^\infty 2^{j(\alpha p -1)}\sum_{i=1}^{2^j} \|R^Y - R^{\tilde{Y}}\|_{\frac{1}{2\alpha}\var;[\frac{i-1}{2^j},\frac{i}{2^j}]}^{\frac{p}{2}}\|X\|_{\frac{1}{\alpha}\var; [\frac{i-1}{2^j},\frac{i}{2^j}]}^{\frac{p}{2}} \\
    &\qquad\le\sum_{j=0}^\infty 2^{j(\alpha p -1)} \sum_{i=1}^{2^j}\|\X\|_{W^\alpha_p; [\frac{i-1}{2^j},\frac{i}{2^j}]}^{\frac{p}{2}}2^{-j\frac{\alpha p -1}{2}}\|R^Y - R^{\tilde{Y}}\|_{\tilde{V}^{2\alpha}_{\frac{p}{2}};[\frac{i-1}{2^j},\frac{i}{2^j}]}^{\frac{p}{2}}2^{-j(\alpha p -1)}\\
    &\qquad\lesssim \sum_{j=0}^\infty 2^{-j\frac{\alpha p -1}{2}}\|\X\|_{W^\alpha_p; [0,1]}^{\frac{p}{2}}\sum_{i=1}^{2^j}\|R^Y - R^{\tilde{Y}}\|_{\tilde{V}^{2\alpha}_{\frac{p}{2}};[\frac{i-1}{2^j},\frac{i}{2^j}]}^{\frac{p}{2}} \\
    &\qquad\lesssim \|\X\|_{W^\alpha_p; [0,1]}^{\frac{p}{2}}\|R^Y - R^{\tilde{Y}}\|_{\tilde{V}^{2\alpha}_{\frac{p}{2}};[0,1]}^{\frac{p}{2}}.
  \end{align*}
  Thus, for $F^1_{s,t} := \|R^Y - R^{\tilde{Y}}\|_{\frac{1}{2\alpha}\var;[s,t]}\|X\|_{\frac{1}{\alpha}\var;[s,t]}$, we obtain that
  $$
    \|F^1\|_{\hat{W}^{2\alpha}_{\frac{p}{2}}} \lesssim  \|\X\|_{W^\alpha_p; [0,1]}\|R^Y - R^{\tilde{Y}}\|_{\tilde{V}^{2\alpha}_{\frac{p}{2}};[0,1]}.
  $$
  Applying the same reasoning to $F^2_{s,t}:= \|R^{\tilde{Y}}\|_{\frac{1}{2\alpha}\var;[s,t]}|X_{s,t} - \tilde{X}_{s,t}|$, $F^3_{s,t}:= \|\tilde{Y}^\prime\|_{\frac{1}{\alpha}\var;[s,t]}\|\mathbb{X} - \tilde{\mathbb{X}}\|_{\frac{1}{2\alpha}\var;[s,t]}$ and $F^4_{s,t}:= \|Y^\prime - \tilde{Y}^\prime\|_{\frac{1}{\alpha}\var;[s,t]}\|\mathbb{X}\|_{\frac{1}{2\alpha}\var;[s,t]}$ and noting that $|h^Y_{s,t} - h^{\tilde{Y}}_{s,t}| \lesssim \sum_{i=1}^4 F^i_{s,t}$, we can conclude that
  $$
    \|h^Y - h^{\tilde{Y}}\|_{\hat{W}^{2\alpha}_{\frac{p}{2}}} \lesssim \|R^Y - R^{\tilde{Y}}\|_{\hat{W}^{2\alpha}_{\frac{p}{2}}}  
    + \hat{\rho}_{W^\alpha_p}(\X, \tilde{\X}),
  $$
  which in turn implies that $\|R^{I_{\X}(Y)}-  R^{I_{\tilde{\X}}(\tilde{Y})}\|_{\hat{W}^{2\alpha}_{\frac{p}{2}}} \lesssim \|R^Y - R^{\tilde{Y}}\|_{\tilde{V}^{2\alpha}_{\frac{p}{2}}}  + \|Y^\prime - \tilde{Y}^\prime\|_{W^\alpha_p} + \hat{\rho}_{W^\alpha_p}(\X, \tilde{\X})$. A similar calculation also provides a similar bound for $\|R^{I_{\X}(Y)}-  R^{I_{\tilde{\X}}(\tilde{Y})}\|_{\tilde{V}^{2\alpha}_{\frac{p}{2}}}$, which completes the proof of (ii).
\end{proof}

\begin{remark}%\label{remark: why use 2alpha and p/2}
  The proof of the Lemma~\ref{lemma: stability result of rough integrals in Sobolev setting} illustrates the reason why we choose the discrete Sobolev norm $\|\cdot\|_{\hat{W}^{2\alpha}_{\frac{p}{2}}}$ instead of $\|\cdot\|_{\hat{W}^{\alpha}_{p}}$ in Definition~\ref{def: Sobolev controlled rough paths} because in general one only has
  $$
    \|R^{I_{\X}(Y)}-  R^{I_{\tilde{\X}}(\tilde{Y})}\|_{\hat{W}^{\alpha}_{p}} \lesssim \|R^Y - R^{\tilde{Y}}\|_{\tilde{V}^{2\alpha}_{\frac{p}{2}}}^{\frac{1} {2}} + \|Y^\prime - \tilde{Y}^\prime\|_{W^\alpha_p}
    + \hat{\rho}_{W^\alpha_p}(\X, \tilde{\X})^{\frac{1}{2}},
  $$
  so that we do not have a (local) Lipschitz estimates. 
  
  The same regularity condition for the second order term $\mathbb{X}$ appears in the framework of paracontrolled distributions when working with Sobolev spaces, see \cite[Definition~5.1]{Promel2016}. 
\end{remark}

\begin{remark}
  Recall that the rough path integration coincides with the classical Young integration if $\alpha > 1/2$. For the Young integral is well-known that the integration operator is continuous with to the Sobolev distance, see, e.g., \cite{Kamont1994} and \cite{Zahle1998,Zahle2001}. This in line with Lemma~\ref{lemma: stability result of rough integrals in Sobolev setting}: In the case $\alpha > 1/2$ the second order term $\mathbb{X}$ does not appear, therefore, the Sobolev distance $\hat{\rho}_{W^\alpha_p}$ can be equivalently defined in its integral form, which dominates the distance $\rho_{\tilde{V}^\alpha_p}$, see \cite[Corollary~2.12]{Friz2018}. However, for the rough path distances we (currently) cannot avoid the use of $\rho_{\tilde{V}^\alpha_p}$, see also Remark~\ref{rmk:inhomogeneous rough path distances} below.
\end{remark}

Controlled paths of Sobolev type are also stable under compositions of smooth functions. For $n\in \N$ let $C^n_b(\R^e;\mathcal{L}(\R^d;\R^e))$ be the space of $n$-times continuously differentiable functions $f\colon \R^e\to \mathcal{L}(\R^d;\R^e)$ such that $f$ and its derivatives of up to order $n$ are bounded.

\begin{lemma}\label{lemma: stability results for compositions of smooth functions of Sobolev controlled rough paths}
  Let $F \in C^3_b(\R^e;\mathcal{L}(\R^d;\R^e))$ and $(Y,Y^\prime) \in \mathcal{D}^{\alpha,p}_{\X}([0,1];\R^e)$. Then, one has:
  \begin{enumerate}
	\item[(i)] $(F(Y),F(Y)^\prime):= (F(Y), DF(Y)Y^\prime) \in \mathcal{D}^{\alpha,p}_{\X}([0,1];\mathcal{L}(\R^d,\R^e))$.
	\item[(ii)] If $\tilde{\X}$ is another rough path in $W^\alpha_p([0,1];G^2(\R^d))$ and $(\tilde{Y},\tilde{Y}^\prime) \in \mathcal{D}^{\alpha,p}_{\tilde{\X}}([0,T];\R^e)$, then we have the following locally uniform estimates
	\begin{align*}
      &\|R^{F(Y)} - R^{F(\tilde{Y})}\|_{\tilde{V}^{2\alpha}_{\frac{p}{2}}}+ \|R^{F(Y)} - R^{F(\tilde{Y})}\|_{\hat{W}^{2\alpha}_{\frac{p}{2}}} \\
	  &\qquad\lesssim \|R^Y - R^{\tilde{Y}}\|_{\tilde{V}^{2\alpha}_{\frac{p}{2}}} + \|R^Y - R^{\tilde{Y}}\|_{\hat{W}^{2\alpha}_{\frac{p}{2}}}  
	  + \|Y^\prime - \tilde{Y}^\prime\|_{W^\alpha_p} + \rho_{\tilde{V}^\alpha_p}(\X, \tilde{\X}) + \hat{\rho}_{W^\alpha_p}(\X, \tilde{\X})
	\end{align*}
	where $R^{F(Y)}$ and $R^{F(\tilde{Y})}$ are the remainder terms of $(F(Y),F(Y)^\prime)$ and $(F(\tilde{Y}),F(\tilde{Y})^\prime)$, respectively.
  \end{enumerate}
\end{lemma}

\begin{proof}
  The proof follows by very similar arguments as in the proof of Lemma~\ref{lemma: stability result of rough integrals in Sobolev setting}, which can be adapted to the present setting without further difficulties.
\end{proof}

The stability results (Lemma~\ref{lemma: stability result of rough integrals in Sobolev setting} and~\ref{lemma: stability results for compositions of smooth functions of Sobolev controlled rough paths}) allow to apply a Banach fixed point argument to show that differential equations driven by Sobolev rough paths along smooth enough vector fields admit a unique solution of the same Sobolev regularity as the driving signals. Moreover, the solution depends continuously on the driving signals in a locally Lipschitz manner. We summarize these facts in the next theorem:

\begin{theorem}\label{thm: Lipschitz continuity of Ito map for 2 levels RDE}
  Suppose $\X$ is a rough path in $W^\alpha_p([0,1];G^2(\R^d))$ and $V \in  C^3_b(\R^e; \mathcal{L}(\R^d;\R^e))$. Then, the rough differential equation
  \begin{equation*}
    Y_t = y_0 + \int_0^t V(Y_s) \dd \X_s,\quad t\in[0,1],
  \end{equation*}
  admits a unique solution $Y \in W^\alpha_p([0,1];\R^e)$. Furthermore, If $\tilde{\X}$ is another rough path in $W^\alpha_p([0,1];G^2(\R^d))$ and $\tilde{Y}$ is the solution to the differential equation driven by $\tilde{\X}$ along $V$ with initial value $y_0$, then it holds that
  \begin{equation*}
    \|Y - \tilde{Y}\|_{W^\alpha_p} \lesssim \rho_{\tilde{V}^\alpha_p}(\X, \tilde{\X}) + \hat{\rho}_{W^\alpha_p}(\X, \tilde{\X}),
  \end{equation*}
  where the proportionality constant only depends on $p$, $\alpha$, $\X$, $\tilde{\X}$ and $V$.
\end{theorem}

As Theorem~\ref{thm: Lipschitz continuity of Ito map for 2 levels RDE} can also be derived as a special case of Theorem~\ref{thm:Ito-Lyons map continuity}, we only outline here the main steps of the proof. However, in the present level-$2$ setting it is more transparent to see why $\rho_{\tilde{V}^\alpha_p} + \hat{\rho}_{W^\alpha_p}$ appear in our stability estimates, in particular, in the local Lipschitz continuity of the map associated to differential equations driven by Sobolev rough paths.

\begin{proof}
  Let $\Phi^V$ be the solution mapping defined on $\mathcal{D}^{\alpha,p}_{\X}([0,1];\R^e)$ into itself, which is given by
  $$
    \Phi^V((Y,Y^\prime)) := \bigg(y_0 + \int V(Y) \dd \X, V(Y)\bigg).
  $$
  By Lemma~\ref{lemma: stability result of rough integrals in Sobolev setting} and \ref{lemma: stability results for compositions of smooth functions of Sobolev controlled rough paths} it is straightforward to check that $\Phi^V$ is a local contraction, and therefore the rough differential equation admits a unique local solution. Then a routine argument in theory of differential equations allows us to paste local solutions together to get a unique global solution. The estimate of $\|Y - \tilde{Y}\|_{W^\alpha_p}$ follows then from the corresponding estimates of the remainder terms in Lemma~\ref{lemma: stability result of rough integrals in Sobolev setting} and \ref{lemma: stability results for compositions of smooth functions of Sobolev controlled rough paths}. We note that every estimates contains the term $\rho_{\tilde{V}^\alpha_p}(\X, \tilde{\X}) + \hat{\rho}_{W^\alpha_p}(\X, \tilde{\X})$. For more details we refer the reader to \cite[Chapter~8]{Friz2014}. Although the setup therein is the H\"older case, one can copy all proofs verbatim to the current Sobolev setting by replacing the inhomogenous H\"older metric through the mixed type metric $\rho_{\tilde{V}^\alpha_p} + \hat{\rho}_{W^\alpha_p}$.
\end{proof}

\begin{remark}
  In the case $\alpha \in (1/3, 1/2)$ the continuity of the It\^o--Lyons map was established in~\cite{Promel2016} also in a Sobolev setting based on the notion of paracontrolled distributions but not on classical rough path spaces. The paracontrolled distribution approach avoids the use of the sewing lemma but does not directly extend to less regular driving signals.
\end{remark}

\section{Rough differential equations driven by Sobolev rough paths}\label{sec: Sobolev RDE}

We consider the controlled differential equation
\begin{equation}\label{eq:controlled diff equa}
  \d Y_t = V(Y_t)\dd X_t, \quad Y_0 = y_0, \quad t \in [0,1],
\end{equation}
for a driving signal $X \in C^{r\var}([0,1];\R^d)$, an initial value $y_0 \in \R^e$ and a vector field $V = (V_1,\dots,V_d)\colon \R^e \rightarrow \L(\R^d;\R^e)$. Let $\Lip^\alpha:=\Lip^\alpha ( \mathbb{R}^e;\mathcal{L}(\mathbb{R}^d;\mathbb{R}^e))$ be the space of all $\alpha$-Lipschitz continuous functions $V\colon \mathbb{R}^e \to \mathcal{L}(\mathbb{R}^d,\mathbb{R}^e)$ in the sense of E. Stein for $\alpha >0$, equipped with the usual norm $| \cdot |_{\Lip^\alpha}$, see \cite[Definition~10.2]{Friz2010}. 

As discussed in the Introduction, if $r>2$, it is not sufficient to take ``only" a $\R^d$-valued path~$X$ as input to the system~\eqref{eq:controlled diff equa} in order to develop a pathwise solution theory. Therefore, we require in the following the driving signal to be a rough path~$\X$. For a given weakly geometric rough path $\X \in C^{r\var}([0,1]; G^{[r]}(\R^d))$,  $Y \in C([0,1];\R^e)$ is said to be a solution to the \textit{rough differential equation}
\begin{equation}\label{eq:RDE}
  \d Y_t = V(Y_t)\dd \X_t, \quad Y_0 = y_0, \quad t \in [0,1],
\end{equation}
if there exist a sequence $(X^n) \subset C^{1\var}([0,1];\R^d)$ such that 
$$
  \lim_{n \rightarrow \infty}\sup_{0\le s \le t \le T}d_{cc}(S_{[r]}(X^n)_{s,t}, \X_{s,t}) = 0, \quad \sup_{n}\|S_{[r]}(X^n)\|_{r\var} < +\infty,
$$
and the corresponding solutions $Y^n$ to equation~\eqref{eq:controlled diff equa} converge uniformly on $[0,T]$ to~$Y$ as $n \rightarrow \infty$, cf. \cite[Definition~10.17]{Friz2010}. By \cite[Theorem~10.14 and Corollary~10.15]{Friz2010}, given a rough path $\X \in C^{r\var}([0,1]; G^{[r]}(\R^d))$ and a vector field $V \in \Lip^{\gamma - 1}$ with $\gamma > r \ge 1$, there exists a solution~$Y$ to the equation~\eqref{eq:RDE} such that for any $[s,t]\subset [0,T]$, 
\begin{equation}\label{eq:Euler approx}
  |Y_t -Y_s - \mathcal{E}_V(Y_s,\X_{s,t})| \lesssim (|V|_{\Lip^{\gamma -1}}\|\X\|_{r\var;[s,t]})^\gamma,
\end{equation}
where $\mathcal{E}_V(Y_s,\X_{s,t})$ denotes the step-$[r]$ Euler scheme (cf. \cite[Definition~10.1]{Friz2010}), namely,
\begin{equation}\label{eq:expression of Euler scheme}
  \mathcal{E}_V(Y_s,\X_{s,t}) := \sum_{k = 1}^{[r]}\sum_{i_1,\dots,i_k \in \{1,\dots,d\}}V_{i_1}\dots V_{i_k}I(Y_s)\pi_k(\X_{s,t})^{i_1,\dots,i_k},
\end{equation}
where $I$ is the identity map on $\R^e$ and $\pi_k(\X_{s,t})^{i_1,\dots,i_k}$ denotes the $(i_1,\dots,i_k)$-component of $\pi_k(\X_{s,t}) \in (\R^{d})^{\otimes k}$. 

Instead of using the classical notation of weakly geometric rough paths of finite $r$-variation, we shall consider the driving signal~$\X$ of the controlled differential equation~\eqref{eq:RDE} to be a Sobolev rough path in $W^{\alpha}_p([0,1];G^{[\frac{1}{\alpha}]}(\R^d))$ with $\alpha \in (0,1)$ and $p \in (1,+\infty]$ such that $\alpha > \frac{1}{p}$, cf. Definition~\ref{def:Sobolev geometric rough path}. From Sobolev embedding theorems, see e.g. \cite[Theorem~2]{Friz2006}, we know that $\X$ still belongs to $C^{r\var}([0,1];G^{[r]}(\R^d))$ with $r := \frac{1}{\alpha}$. Hence, if the vector field~$V$ in \eqref{eq:RDE} belongs to $\Lip^{\gamma - 1}$ with $\gamma > r \ge 1$, then by classical results from rough path theory, as stated above, there exists a solution $Y \in C^{r\var}([0,1];\R^e)$ to the rough differential equation~\eqref{eq:RDE}. The following proposition shows that in this case we even obtain the solution~$Y$ to be of Sobolev regularity. Namely, $Y$ has exactly the same Sobolev regularity as the driving signal~$\X$. 

\begin{proposition}\label{prop:Young-Love inequality RDE setting}
  Let $\alpha \in (0,1)$ and $p \in (1,+\infty]$ be such that $\alpha > 1/p$. Suppose that $\X \in W^{\alpha}_p([0,1];G^{[\frac{1}{\alpha}]}(\R^d))$ and $V \in \Lip^{\gamma - 1}$ for some $\gamma > 1/\alpha $. Then, for any initial condition $y_0 \in \R^e$ there exists a solution~$Y$ to the rough differential equation~\eqref{eq:RDE} with $Y_0 = y_0$. Moreover, there exists a continuous increasing function $f\colon \R_+ \to \R_+$ such that for all $\X \in W^{\alpha}_p([0,1];G^{[\frac{1}{\alpha}]}(\R^d))$ with $\sup_{t \in [0,1]}\|\X_t\|_{cc} \leq M$, one has
  \begin{equation*}%\label{eq:Young-Love inequality for rough path}
    \|Y\|_{W^{\alpha}_p} \lesssim f(M)\Big(|V|_{\Lip^{\gamma - 1}}\|\X\|_{W^{\alpha}_p} + (|V|_{\Lip^{\gamma - 1}}\|\X\|_{W^{\alpha}_p})^\gamma\Big).
  \end{equation*}
\end{proposition}

\begin{proof}[Proof of Proposition~\ref{prop:Young-Love inequality RDE setting}]
  Since we have that $\X \in C^{\frac{1}{\alpha}\var}([0,1]; G^{[\frac{1}{\alpha}]}(\R^d))$, see \cite[Theorem~2]{Friz2006}, there exists a solution $Y$ to the rough differential equation~\eqref{eq:RDE} with $Y_0 = y_0$ and~\eqref{eq:Euler approx} holds. As a consequence, for every $j \in \mathbb{N}$, one has
  \begin{equation*}
    \sum_{k = 1}^{2^j} |Y_{k2^{-j}} - Y_{(k-1)2^{-j}}|^p \lesssim \sum_{k = 1}^{2^j}|\mathcal{E}_V(Y_{(k-1)2^{-j}},\X_{(k-1)2^{-j},k2^{-j}})|^p + \sum_{k = 1}^{2^j}\|\X\|_{\frac{1}{\alpha}\var;[(k-1)2^{-j},k2^{-j}]}^{\gamma p}.
  \end{equation*}
  From the expression~\eqref{eq:expression of Euler scheme} we can deduce that 
  \begin{equation*}
    |\mathcal{E}_V(Y_{(k-1)2^{-j}},\X_{(k-1)2^{-j},k2^{-j}})| \lesssim |V|_{\Lip^{\gamma - 1}}|\X_{(k-1)2^{-j},k2^{-j}}|.
  \end{equation*}
  Furthermore, by \cite[(7.22)]{Friz2010} we have 
  \begin{equation*}
    |\X_{(k-1)2^{-j},k2^{-j}}| \lesssim \max\Big(1,\sup_{t \in [0,1]}\|\X_t\|_{cc}^{[\frac{1}{\alpha}]}\Big)\rho(\X_{k2^{-j}}, \X_{(k-1)2^{-j}}).
  \end{equation*}
  Hence, by assumptions we obtain that
  \begin{equation}\label{eq:bound for Euler scheme}
    |\mathcal{E}_V(Y_{(k-1)2^{-j}},\X_{(k-1)2^{-j},k2^{-j}})| \lesssim \rho(\X_{k2^{-j}}, \X_{(k-1)2^{-j}}). 
  \end{equation}
  On the other hand, by \cite[Corollary~2.12]{Friz2018} we get 
  \begin{equation}\label{eq:bound for variation norm}
    \|\X\|_{\frac{1}{\alpha}\var;[(k-1)2^{-j},k2^{-j}]}^{p} \lesssim \|\X\|_{W^{\alpha}_p;[(k-1)2^{-j},k2^{-j}]}^p 2^{-j(\alpha p - 1)}.
  \end{equation}
  Inserting~\eqref{eq:bound for Euler scheme} and~\eqref{eq:bound for variation norm} into the above estimate, we arrive at
  \begin{align*}
    \sum_{k = 1}^{2^j} |Y_{k2^{-j}} - Y_{(k-1)2^{-j}}|^p \lesssim \sum_{k = 1}^{2^j}\rho(\X_{k2^{-j}}, \X_{(k-1)2^{-j}})^p + \sum_{k = 1}^{2^j}\Big(\|\X\|_{W^{\alpha}_p;[(k-1)2^{-j},k2^{-j}]}^p 2^{-j(\alpha p - 1)}\Big)^\gamma.
  \end{align*}
  It follows that
  \begin{align}\label{eq:Rosenbaum norm}
    \sum_{j \ge 0} 2^{j(\alpha p - 1)}\sum_{k = 1}^{2^j} |Y_{k2^{-j}} - Y_{(k-1)2^{-j}}|^p \lesssim& \sum_{j \ge 0} 2^{j(\alpha p - 1)}\sum_{k = 1}^{2^j}\rho(\X_{k2^{-j}}, \X_{(k-1)2^{-j}})^p \nonumber\\
    &+ \sum_{j \ge 0} 2^{j(\alpha p - 1)}\sum_{k = 1}^{2^j}\Big(\|\X\|_{W^{\alpha}_p;[(k-1)2^{-j},k2^{-j}]}^p 2^{-j(\alpha p - 1)}\Big)^\gamma.
  \end{align}
  Applying Theorem~\eqref{thm:Sobolev norm equivalent}, for the Euclidean metric $\rho$, to the first term in the right-hand side of inequality~\eqref{eq:Rosenbaum norm}, we conclude that
  \begin{equation*}  
    \sum_{j \ge 0} 2^{j(\alpha p - 1)}\sum_{k = 1}^{2^j}\rho(\X_{k2^{-j}}, \X_{(k-1)2^{-j}})^p \lesssim \iint_{[0,1]^2}\frac{\rho(\X_u,\X_v)^p}{|v - u|^{\alpha p+1}} \dd u \dd v.
  \end{equation*}
  Invoking that $\rho(g,h) \lesssim d_{cc}(g,h)$ locally uniformly on $G^{[\frac{1}{\alpha}]}(\R^d)$, we can further deduce that
  \begin{equation*}
    \sum_{j \ge 0} 2^{j(\alpha p - 1)}\sum_{k = 1}^{2^j}\rho(\X_{k2^{-j}}, \X_{(k-1)2^{-j}})^p \lesssim \iint_{[0,T]^2}\frac{d_{cc}(\X_u,\X_v)^p}{|v - u|^{\alpha p+1}} \dd u \dd v = \|\X\|_{W^{\alpha}_p}^p
  \end{equation*}
  and thus
  \begin{equation}\label{eq:bound for the first term}
    \sum_{j \ge 0} 2^{j(\alpha p - 1)}\sum_{k = 1}^{2^j}\rho(\X_{k2^{-j}}, \X_{(k-1)2^{-j}})^p \lesssim \|\X\|_{W^{\alpha}_p}^p.
  \end{equation}
  Let us now turn to the second term in the right-hand side of~\eqref{eq:Rosenbaum norm}. Since $\gamma >\frac{1}{\alpha} > 1$, the elementary inequality $\sum|a_i|^\gamma \leq (\sum |a_i|)^{\gamma}$ implies that
  \begin{align*}
    \sum_{j \ge 0} 2^{j(\alpha p - 1)}\sum_{k = 1}^{2^j}\Big(\|\X\|_{W^{\alpha}_p;[(k-1)2^{-j},k2^{-j}]}^p & 2^{-j(\alpha p - 1)}\Big)^\gamma\\
    &\lesssim \Big(\sum_{j \ge 0} \sum_{k = 1}^{2^j}2^{-j(\alpha p - 1)(1-\frac{1}{\gamma})}  \|\X\|_{W^{\alpha}_p;[(k-1)2^{-j},k2^{-j}]}^p \Big)^\gamma.
  \end{align*}
  Since $1 - \frac{1}{\gamma} > 0$ and $\alpha p - 1 > 0$, using the super-additivity of the control function $\omega(s,t) := \|\X\|_{W^{\alpha}_p;[s,t]}^p$, we can immediately deduce that
  \begin{equation}\label{eq:bound for the second term}
    \sum_{j \ge 0} 2^{j(\alpha p - 1)}\sum_{k = 1}^{2^j}\Big(\|\X\|_{W^{\alpha}_p;[(k-1)2^{-j},k2^{-j}]}^p 2^{-j(\alpha p - 1)}\Big)^\gamma \lesssim  \|\X\|_{W^{\alpha}_p}^{\gamma p}.
  \end{equation}
  Inserting the bounds~\eqref{eq:bound for the first term} and~\eqref{eq:bound for the second term} into inequality~\eqref{eq:Rosenbaum norm} and noting that the left-hand side of~\eqref{eq:Rosenbaum norm} is equivalent to the $p$-th power of the $W^{\alpha}_p$-norm of $Y$ due to Theorem~\ref{thm:Sobolev norm equivalent}, we finally obtain that
  \begin{equation*}
    \|Y\|_{W^{\alpha}_p} \lesssim \|\X\|_{W^{\alpha}_p} + \|\X\|_{W^{\alpha}_p}^{\gamma},
  \end{equation*}
  where the proportionality constant depends continuously on $M$ and is increasing in $M$ (in fact, we may choose $f(M) := \max(1, M^{[\frac{1}{\alpha}]})$). This completes the proof.
\end{proof}

\section{Continuity of the It\^o--Lyons map on Sobolev spaces}\label{sec: Ito map}

If the vector field~$V$ belongs even to $\Lip^\gamma$ rather than $\Lip^{\gamma - 1}$ for $\gamma >1/\alpha$, then classical results from rough path theory (see, e.g., \cite[Theorem~10.26]{Friz2010}) imply the uniqueness of the solution~$Y$ to the rough differential equation~\eqref{eq:RDE}. Recalling that the solution~$Y$ is an element of $W^\alpha_p([0,T];\R^e)$ by Proposition~\ref{prop:Young-Love inequality RDE setting}, the It\^o--Lyons map~$\Phi$ given by
\begin{equation}\label{eq:Ito-Lyons map in Sobolev setup}
  \Phi \colon \R^e \times \Lip^\gamma \times W^{\alpha}_p([0,1];G^{[\frac{1}{\alpha}]}(\R^d)) \rightarrow W^{\alpha}_p([0,1];\R^e) \quad \text{via} \quad  \Phi(y_0,V,\X) := Y,
\end{equation}
where $Y$ denotes the unique solution to rough differential equation~\eqref{eq:RDE} given the input $(y_0,V,\X)$, is well-defined.

One of the central results of rough path theory is the local Lipschitz continuity of the It\^o--Lyons map, which, of course, crucially depends on the chosen topology. We now establish the local Lipschitz continuity of the It\^o--Lyons map acting on the space of Sobolev rough paths, as defined in~\eqref{eq:Ito-Lyons map in Sobolev setup}.

\begin{theorem}\label{thm:Ito-Lyons map continuity}
  Let $\alpha \in (0,1)$, $\gamma >1$ and $p \in (1,+\infty)$ be such that $\alpha > 1/p$ and $\gamma > 1/\alpha $. Then, the It\^o--Lyons map $\Phi$ as defined in~\eqref{eq:Ito-Lyons map in Sobolev setup} is locally Lipschitz continuous with respect to the initial value, vector field and the driving signal, that is, for $y^i_0 \in \R^e$, $V^i \in \Lip^\gamma$ and $\X^i \in W^\alpha_p([0,1];G^{[\frac{1}{\alpha}]}(\R^d))$ satisfying
  \begin{equation*}
    \|\X^i\|_{W^{\alpha}_p} \le b \quad \text{and} \quad |V^i |_{\Lip^\gamma} \le l, \quad i=1,2,
  \end{equation*}
  for some $b,l>0$, with corresponding solution $Y^i = \Phi(y^i_0,V^i,\X^i)$, there exists a constant $C = C(b,l,\gamma,\alpha, p,T) \ge 1$ such that
  \begin{equation*}
    \|Y^1 - Y^2\|_{W^{\alpha}_p} \le C\Big( |V^1 - V^2 |_{\Lip^{\gamma - 1}} + |y^1_0 - y^2_0| + \hat{\rho}_{W^{\alpha}_p}(\X^1, \X^2) + \rho_{\tilde{V}^\alpha_p}(\X^1,\X^2)\Big).
  \end{equation*}
\end{theorem}

\begin{proof} 
  A careful inspection of the proof of \cite[Theorem~10.26]{Friz2010} (see also Remark \ref{rmk: modifying the proof of theorem 10.26 in FV18} below) reveals that if $\omega$ is a control function on $\Delta$ and $\omega^\prime$ is a non-negative function on $\Delta$ such that
  $$
    \|\X^i\|_{\frac{1}{\alpha}\text{-}\omega}:= \sup_{0\le s \le t \le 1}\frac{\|\X^i_{s,t}\|_{cc}}{\omega(s,t)^{\alpha}} \le 1
    \quad \text{and}\quad
    \|\X^i\|_{\frac{1}{\alpha}\text{-}\omega^\prime}:= \sup_{0\le s \le t \le 1}\frac{\|\X^i_{s,t}\|_{cc}}{\omega^\prime(s,t)^{\alpha}}\le 1, 
  $$
  for $i=1,2$, then for any $s<t$ in $[0,1]$,
  \begin{align}\label{eq:differnece between two solutions}
    \begin{split}
    |Y^1_{s,t} -& Y^2_{s,t}|\\
    \lesssim &(l|y^1_0 - y^2_0| + |V^1 - V^2|_{\Lip^{\gamma - 1}} + l\rho_{\frac{1}{\alpha}\text{-}\omega^\prime}(\X^1,\X^2))\omega^\prime(s,t)^\alpha \exp(Cl\omega^\prime(s,t) + Cl^{\frac{1}{\alpha}}\omega(0,1))  \\
    &+ (l|y^1_0 - y^2_0| + |V^1 - V^2|_{\Lip^{\gamma - 1}} + l\rho_{\frac{1}{\alpha}\text{-}\omega}(\X^1,\X^2))l^{\gamma - 1}\omega(s,t)^{\gamma \alpha}\exp( Cl^{\frac{1}{\alpha}}\omega(0,1)), 
    \end{split}
  \end{align}
  where $\rho_{\frac{1}{\alpha}\text{-}\omega}(\X^1,\X^2) := \sum_{k = 1,\dots,[\frac{1}{\alpha}]}\sup_{0\le s \le t \le 1}\frac{|\pi_k(\X^1_{s,t} - \X^2_{s,t})|}{\omega(s,t)^{\alpha k}}$ and the same expression holds for $\rho_{\frac{1}{\alpha}\text{-}\omega^\prime}(\X^1,\X^2)$. Let us define
  \begin{equation}\label{eq:control function 1}
    \omega(s,t) := \|\X^1\|^{\frac{1}{\alpha}}_{\frac{1}{\alpha}\var;[s,t]} + \|\X^1\|^{\frac{1}{\alpha}}_{\frac{1}{\alpha}\var;[s,t]} + \sum_{k = 1}^{[\frac{1}{\alpha}]}\omega^{(k)}_{\X^1,\X^2}(s,t),
  \end{equation}
  where $\omega^{(k)}_{\X^1,\X^2}(s,t) := \Big(\frac{\rho^{(k)}_{\frac{1}{\alpha}\var;[s,t]}(\X^1,\X^2)}{\rho^{(k)}_{\tilde{V}^{\alpha}_p}(\X^1,\X^2)}\Big)^{\frac{1}{\alpha k}}$ and $\rho^{(k)}_{\frac{1}{\alpha}\var}$ is the inhomogeneous variation metric defined in \cite[Definition~8.6]{Friz2010}.
  Furthermore, we set
  \begin{equation*}%\label{eq:control function 2}
    \omega^{\prime}(s,t) := \|\X^1_{s,t}\|_{cc}^{\frac{1}{\alpha}} + \|\X^2_{s,t}\|_{cc}^{\frac{1}{\alpha}} + \sum_{k = 1}^{[\frac{1}{\alpha}]}\omega^{\prime,(k)}_{\X^1,\X^2}(s,t)
  \end{equation*}
  with  $\omega^{\prime,(k)}_{\X^1,\X^2}(s,t) := \Big(\frac{|\pi_k(\X^1_{s,t} - \X^2_{s,t})|}{\hat{\rho}^{(k)}_{W^{\alpha}_p}(\X^1,\X^2)}\Big)^{\frac{1}{\alpha k}}$. By definition, we see that for such $\omega$ and $\omega^\prime$ it holds that $\|\X^i\|_{\frac{1}{\alpha}\text{-}\omega} \le 1$ and $ \|\X^i\|_{\frac{1}{\alpha}\text{-}\omega^\prime}\le 1$ for $i=1,2$. Moreover, since
  $$
    |\pi_k(\X^1_{s,t} - \X^2_{s,t})|\leq \frac{\rho^{(k)}_{\frac{1}{\alpha}\var;[s,t]}(\X^1,\X^2)}{\rho^{(k)}_{\tilde{V}^{\alpha}_p}(\X^1,\X^2)} \rho^{(k)}_{\tilde{V}^{\alpha}_p}(\X^1,\X^2)\leq \omega(s,t)^{\alpha k}\rho^{(k)}_{\tilde{V}^{\alpha}_p}(\X^1,\X^2),
  $$
  we indeed have $\rho_{\frac{1}{\alpha}\text{-}\omega}(\X^1,\X^2)\le \rho_{\tilde{V}^{\alpha}_p}(\X^1,\X^2)$. By the same reasoning we can also deduce that $\rho_{\frac{1}{\alpha}\text{-}\omega^\prime}(\X^1,\X^2)\le \hat{\rho}_{W^{\alpha}_p}(\X^1,\X^2)$. Although $\omega^\prime$ is not a control function, it holds that $\omega^\prime(s,t) \le \omega(s,t)$ for all $s<t$ in $[0,1]$. Hence, we can bound the $\omega^\prime(s,t)$ appearing in the exponential function in~\eqref{eq:differnece between two solutions} by $\omega(0,1)$. All above observations allow us to reduce estimate~\eqref{eq:differnece between two solutions} to
  \begin{align*}
    |Y^1_{s,t} - Y^2_{s,t}| \lesssim &\Big(l|y^1_0 - y^2_0| + |V^1 - V^2|_{\Lip^{\gamma - 1}} + l\hat{\rho}_{W^{\alpha}_p}(\X^1,\X^2) + l\rho_{\tilde{V}^{\alpha}_p}(\X^1,\X^2) \Big)\exp(Cl^{\frac{1}{\alpha}}\omega(0,1)) \\
    &\times (\omega^\prime(s,t)^\alpha  + \omega(s,t)^{\gamma \alpha}).
  \end{align*}
  For simplicity we denote 
  $$
    F:= \Big(l|y^1_0 - y^2_0| + |V^1 - V^2|_{\Lip^{\gamma - 1}} + l\hat{\rho}_{W^{\alpha}_p}(\X^1,\X^2) + l\rho_{\tilde{V}^{\alpha}_p}(\X^1,\X^2) \Big)\exp(Cl^{\frac{1}{\alpha}}\omega(0,1)),
  $$ which is a constant independent of $(s,t)$. Then we obtain that
  \begin{align}\label{eq:Rosenbaum norm of difference of two solutions}
    \begin{split}
    \sum_{j \ge 0} 2^{j(\alpha p - 1)}\sum_{i = 1}^{2^j} |Y^1_{s,t} - Y^2_{s,t}|^p \lesssim &F^p\Big(\sum_{j \ge 0}2^{j(\alpha p - 1)}\sum_{i = 1}^{2^j}\omega^\prime((i-1)2^{-j}, i2^{-j})^{\alpha p} \\
    &\qquad\quad + \sum_{j \ge 0}2^{j(\alpha p - 1)}\sum_{i = 1}^{2^j}\omega((i-1)2^{-j}, i2^{-j})^{\gamma \alpha p}\Big).
    \end{split}
  \end{align}
  By definition, we have 
  \begin{align*}
    \omega^\prime((i-1)2^{-j}, i2^{-j})^{\alpha p}
    &\lesssim d_{cc}(\X^1_{(i-1)2^{-j}},\X^1_{i2^{-j}})^p + d_{cc}(\X^2_{(i-1)2^{-j}},\X^2_{i2^{-j}})^p \\
    &\quad+\sum_{k=1}^{[\frac{1}{\alpha}]}|\pi_k(\X^1_{(i-1)2^{-j},i2^{-j}} - \X^2_{(i-1)2^{-j},i2^{-j}})|^{\frac{p}{k}}\hat{\rho}^{(k)}_{W^{\alpha}_p}(\X^1,\X^2)^{-\frac{p}{k}}.
  \end{align*}
  In view of the definition of $\hat{\rho}^{(k)}_{W^{\alpha}_p}$ we observe that 
  $$
    \sum_{j \ge 0}2^{j(\alpha p -1)}\sum_{i=1}^{2^j}|\pi_k(\X^1_{(i-1)2^{-j},i2^{-j}} - \X^2_{(i-1)2^{-j},i2^{-j}})|^{\frac{p}{k}} = \hat{\rho}^{(k)}_{W^{\alpha}_p}(\X^1,\X^2)^{\frac{p}{k}} 
  $$
  and thus
  \begin{align*}
    &\sum_{j \ge 0}2^{j(\alpha p - 1)}\sum_{i = 1}^{2^j}\omega^\prime((i-1)2^{-j}, i2^{-j})^{\alpha p}\\
    &\quad= \sum_{j \ge 0}2^{j(\alpha p - 1)}\sum_{i = 1}^{2^j}d_{cc}(\X^1_{(i-1)2^{-j}},\X_{i2^{-j}})^p + \sum_{j \ge 0}2^{j(\alpha p - 1)}\sum_{i = 1}^{2^j}d_{cc}(\X^2_{(i-1)2^{-j}},\X_{i2^{-j}})^p + [\frac{1}{\alpha}].
  \end{align*}
  By Theorem~\ref{thm:Sobolev norm equivalent} the right-hand side of the above inequality is bounded by the term 
  $$
    C(\|\X^1\|_{W^\alpha_p}^p + \|\X^2\|_{W^\alpha_p}^p) + [\frac{1}{\alpha}]
  $$
  for some constant~$C$ only depending on~$\alpha$ and~$p$, therefore the term 
  $$
    \sum_{j \ge 0}2^{j(\alpha p - 1)}\sum_{i = 1}^{2^j}\omega^\prime((i-1)2^{-j}, i2^{-j})^{\alpha p}
  $$
  is bounded by $C(b^p + 1)$ due to our hypothesis. 

  On the other hand, in view of the definition of $\omega$ (cf.~\eqref{eq:control function 1}), one has
  \begin{align*}
    \sum_{j \ge 0}2^{j(\alpha p - 1)}& \sum_{i = 1}^{2^j}\omega((i-1)2^{-j}, i2^{-j})^{\gamma \alpha p}\\
    &\lesssim \sum_{j \ge 0}2^{j(\alpha p - 1)}\sum_{i = 1}^{2^j}\|\X^1\|_{\frac{1}{\alpha}\var;[(i-1)2^{-j},i2^{-j}]}^{\gamma p} 
    +\sum_{j \ge 0}2^{j(\alpha p - 1)}\sum_{i = 1}^{2^j}\|\X^2\|_{\frac{1}{\alpha}\var;[(i-1)2^{-j},i2^{-j}]}^{\gamma p} \\
    &\quad+ \sum_{k=1}^{[\frac{1}{\alpha}]}\sum_{j \ge 0}2^{j(\alpha p - 1)}\sum_{i = 1}^{2^j}\rho^{(k)}_{\frac{1}{\alpha}\var;[(i-1)2^{-j},i2^{-j}]}(\X^1,\X^2)^{\frac{p}{k}\gamma} \rho^{(k)}_{\tilde{V}^\alpha_p}(\X^1,\X^2)^{-\frac{p}{k}\gamma}.
  \end{align*}
  From the proof of Proposition~\ref{prop:Young-Love inequality RDE setting} we see that the right-hand side of the above inequality is bounded by $C(\|\X^1\|_{W^{\alpha}_p}^p + \|\X^2\|_{W^{\alpha}_p}^p)^\gamma$ for some constant~$C$ only depending on $\alpha,p$ and $\gamma$. 
  
  For the last term, note that in the proof of \cite[Theorem~3.3]{Friz2018} one has 
  $$
    \rho^{(k)}_{\frac{1}{\alpha}\var;[(i-1)2^{-j},i2^{-j}]}(\X^1,\X^2) \leq \rho^{(k)}_{\tilde{V}^{\alpha}_p;[(i-1)2^{-j},i2^{-j}]}(\X^1,\X^2)2^{-j(\alpha - \frac{1}{p})k}.
  $$  Since $\gamma > 1/\delta > 1$ and $\rho^{(k)}_{\tilde{V}^{\alpha}_p;[s,t]}(\X^1,\X^2)^{\frac{p}{k}}$ is super-additive as a function on $\Delta$, we can deduce that
  \begin{align*}
    \sum_{j \ge 0}2^{j(\alpha p - 1)}\sum_{i = 1}^{2^j} & \rho^{(k)}_{\frac{1}{\alpha}\var;[(i-1)2^{-j},i2^{-j}]}(\X^1,\X^2)^{\frac{p}{k}\gamma} \\
    &\lesssim \Big(\sum_{j \ge 0}2^{-j(\alpha p - 1)(1- \frac{1}{\gamma})}\sum_{i = 1}^{2^j}\rho^{(k)}_{\tilde{V}^{\alpha}_p;[(i-1)2^{-j},i2^{-j}]}(\X^1,\X^2)^{\frac{p}{k}} \Big)^\gamma
    \lesssim \rho^{(k)}_{\tilde{V}^{\alpha}_p}(\X^1,\X^2)^{\frac{p}{k}\gamma},
  \end{align*}
  for every $k = 1,\dots,[\frac{1}{\alpha}]$. Hence, we obtain that
  \begin{equation*}
    \sum_{j \ge 0}2^{j(\alpha p - 1)}\sum_{i = 1}^{2^j}\omega((i-1)2^{-j}, i2^{-j})^{\gamma \alpha p} \lesssim b^{p\gamma} + 1.
  \end{equation*}
  Note that from the above estimate we also deduce that $\omega(0,1) \le C(b^{\frac{1}{\alpha}} + 1)$ for some constant~$C$ only depending on $\alpha,p$ and $\gamma$. Now inserting all above estimates into~\eqref{eq:Rosenbaum norm of difference of two solutions} and using Theorem~\ref{thm:Sobolev norm equivalent}, we find that
  \begin{equation*}
    \|Y^1 - Y^2\|_{W^{\alpha}_p} \lesssim C\Big(|V^1 - V^2|_{\Lip^{\gamma - 1}} + |y^1_0 - y^2_0| + \hat{\rho}_{W^{\alpha}_p}(\X^1, \X^2) + \rho_{\tilde{V}^\alpha_p}(\X^1,\X^2) \Big).
  \end{equation*}
\end{proof}

\begin{remark}\label{rmk: modifying the proof of theorem 10.26 in FV18}
  Let us briefly show that how to derive \eqref{eq:differnece between two solutions} from \cite[Theorem 10.26]{Friz2010}. First note that if $\omega^\prime$ is a non-negative function on $\Delta$ such that $\|\mathbf{X}^i\|_{\frac{1}{\alpha}\text{-}\omega^\prime} \le 1$ for $i=1,2$, then for any $s<t$, we can find paths $x^{1,s,t}$ and $x^{2,s,t}$ such that 
  $$
    S_{[\frac{1}{\alpha}]}(x^{i,s,t}) = \X^i_{s,t}, \quad i=1,2,
  $$
  and 
  $$
    \int_s^t |\d x^{i,s,t}_r| \le c_1 \omega^\prime(s,t)^\alpha, \quad i=1,2, \quad \int_s^t |\d x^{1,s,t}_r-\d x^{2,s,t}_r| \le c_1\omega^\prime(s,t)\rho_{\frac{1}{\alpha}\text{-}\omega^\prime;[s,t]}(\X^1,\X^2),
  $$
  for some constant $c_1$ only depending on $\alpha$. This argument is actually contained in the proof of \cite[Theorem~10.26, pp. 234]{Friz2010} but with an additional assumption that $\omega^\prime$ is a control function. However, as this argument is essentially a direct application of \cite[Proposition~7.64]{Friz2010}, one can easily see that the super-additivity of $\omega^\prime$ plays no role therein and what we need is just the bound $\|\mathbf{X}^i\|_{\frac{1}{\alpha}\text{-}\omega^\prime} \le 1$ for $i=1,2$. As a consequence, by \cite[Theorem~3.18]{Friz2010} or \cite[Theorem~10.26, pp. 235--236]{Friz2010}, we can replace the control $\omega(s,t)$ appeared in \cite[(10.24)]{Friz2010} by $\omega^\prime$, and obtain that
  \begin{align*}
    &|\pi_{(V^1)}(s,Y^1_s;x^{1,s,t})_{s,t} - \pi_{(V^2)}(s,Y^2_s;x^{2,s,t})_{s,t}| \\ 
    &\quad\le C (l|y^1_0 - y^2_0| + |V^1 - V^2|_{\Lip^{\gamma - 1}} 
    + l\rho_{\frac{1}{\alpha}\text{-}\omega^\prime}(\X^1,\X^2))\omega^\prime(s,t)^\alpha \exp(Cl\omega^\prime(s,t)),
  \end{align*}
  which is exactly the first summand on the right hand-side of \eqref{eq:differnece between two solutions}. On the other hand, a direct application of the bound for $|\overline \Gamma_{s,t}|$ proved in \cite[Theorem~10.26, pp.~236]{Friz2010} gives 
  $$
    |\overline \Gamma_{s,t}| \le C (l|y^1_0 - y^2_0| + |V^1 - V^2|_{\Lip^{\gamma - 1}} + l\rho_{\frac{1}{\alpha}\text{-}\omega}(\X^1,\X^2))l^{\gamma - 1}\omega(s,t)^{\gamma \alpha}\exp( Cl^{\frac{1}{\alpha}}\omega(0,1)),
  $$
  which is the second summand on the right hand-side of \eqref{eq:differnece between two solutions}. Now the inequality \eqref{eq:differnece between two solutions} follows immediately from the triangle inequality and the relation 
  $$
    |(Y^1_{s,t} - Y^2_{s,t}) - \overline \Gamma_{s,t}| = \Big|\pi_{(V^1)}(s,Y^1_s;x^{1,s,t})_{s,t} - \pi_{(V^2)}(s,Y^2_s;x^{2,s,t})_{s,t}\Big|, 
  $$
  which was established in \cite[Theorem~10.26, pp.~235]{Friz2010}.
\end{remark}

\begin{remark}\label{rmk:inhomogeneous rough path distances}
  One would expect that the inhomogeneous mixed H\"older-variation distance~$\rho_{\tilde{V}^\alpha_p}$ is not needed for the continuity statement of Theorem~\ref{thm:Ito-Lyons map continuity} and that $\rho_{\tilde{V}^\alpha_p}$ is dominated by the inhomogeneous Sobolev distance~$\hat{\rho}_{W^{\alpha}_p}$ as one can observe for the homogeneous Sobolev norms. However, at least if one wants to follow a similar approach as developed in the present work, this would require an extensive study of the inhomogeneous distances: First, it seems to require to generalize the Sobolev-variation embedding theorem for functions $f\colon [0,T]\to E $ provided in \cite{Friz2006} to functions $f\colon \Delta \to E$. Second, similar generalizations seem to be needed for the characterization of non-linear Sobolev spaces \cite{Liu2020} as well as for the embedding results in \cite{Friz2018}. These generalizations are outside the scope of the present article. 
\end{remark}

%\bibliography{quellen}{}
%\bibliographystyle{amsalpha}

\providecommand{\bysame}{\leavevmode\hbox to3em{\hrulefill}\thinspace}
\providecommand{\MR}{\relax\ifhmode\unskip\space\fi MR }
\providecommand{\MRhref}[2]{%
  \href{http://www.ams.org/mathscinet-getitem?mr=#1}{#2}
}
\providecommand{\href}[2]{#2}

\end{document}